\documentclass{amsart}
\usepackage[margin=1in]{geometry}
\usepackage{amsmath, amsfonts, amssymb, amsthm}
\usepackage{graphicx}
\usepackage{hyperref, url}
\usepackage{xcolor}
\usepackage{dsfont}

\newtheorem{thm}{Theorem}[section]
\newtheorem{prop}[thm]{Proposition}

\newtheorem{lemma}[thm]{Lemma}

\newtheorem{preremark}[thm]{Remark}
\newenvironment{remark}{\begin{preremark}\rm}{\medskip \end{preremark}}
\numberwithin{equation}{section}

\newcommand{\R}{\mathbb R}
\newcommand{\eps}{\varepsilon}

\newcommand{\dd} {\; \mathrm{d}}

\DeclareMathOperator{\dv}{div}

\DeclareMathOperator{\trace}{trace}
\DeclareMathOperator{\rank}{rank}

\DeclareMathOperator{\sphere}{sphere}

\newcommand{\n}{{n}}

\title{The Landau equation does not blow up}
\date{\today}
\author{Nestor Guillen and Luis Silvestre}
\thanks{Luis Silvestre is supported by NSF grants DMS-2054888 and DMS-2350263. Nestor Guillen is supported by NSF grant DMS-2144232.}

\begin{document}

\begin{abstract}
We consider solutions to the space-homogeneous Landau equation with a general family of interaction potentials. We prove that their Fisher information is monotone decreasing in time. The class of interaction potentials covered by our result includes the case of the Landau equation with Coulomb interactions. As a consequence of the global boundedness of the Fisher information, we deduce that solutions to the space-homogeneous Landau equation never blow up.
\end{abstract}

\maketitle

\section{Introduction}

The Landau equation was derived in 1936 by Lev Landau \cite{landau1936} to model the evolution of densities of particles performing Coulomb collisions in a plasma. It can be derived as a limit case of the Boltzmann equation when the grazing collisions dominate the evolution. It is one of the central equations in kinetic theory. In the space-homogeneous case, it is given by a very simple formula
\[ f_t = \bar a_{ij} \partial_{ij} f + f^2, \qquad \text{ where } \bar a_{ij} = -\partial_{ij} (-\Delta)^{-2} f.\]
A more general family of equations is usually studied. It has the form
\begin{equation} \label{e:landauequation}
 f_t = q(f),
\end{equation}
where the operator $q(f)$ is given by the formula
\begin{equation} \label{e:landauoperator}
q(f)(v) = \partial_{v_i} \int_{\R^3} \alpha(|v-w|) a_{ij}(v-w) \left( \partial_{v_j} - \partial_{w_j} \right) [f(v) f(w)] \dd w. 
\end{equation}

Here $\alpha : (0,+\infty) \to [0,+\infty)$ is an arbitrary nonnegative function and $a_{ij}(z) = |z|^2 \delta_{ij} - z_i z_j$. We will refer to $\alpha$ as the interaction potential. It is common to study the case $\alpha(r) = r^\gamma$ with $\gamma \in [-3,1]$. The most important case is $\alpha(r) = r^{-3}$ that corresponds to the original Landau equation written above, for charged particles interacting with Coulomb potentials.

For a positive function $f : \R^3 \to (0,\infty)$, its Fisher information is defined by the following expression
\begin{equation} \label{e:fisherinfo}
i(f) := \int_{\R^3} \frac{|\nabla f|^2}f \dd v.
\end{equation}
The formula \eqref{e:fisherinfo} is extended to the case that $f$ has vacuum regions by setting $|\nabla f|^2/f = 0$ at those points.

Our main result is the following.
\begin{thm} \label{t:main}
Let $f : [0,T] \times \R^3 \to [0,\infty)$ be a classical solution to the space-homogeneous Landau equation \eqref{e:landauequation}. Assume that the interaction potential $\alpha$ satisfies, for all $r>0$,
\[ \frac{r |\alpha'(r)|}{\alpha(r)} \leq \sqrt{19},\]
then the Fisher information $i(f)$ is monotone decreasing as a function of time.
\end{thm}

We present Theorem \ref{t:main} as an a-priori estimate for classical solutions. By that, we mean smooth functions $f$ that decay sufficiently fast as $|v| \to \infty$. Our second main theorem tells us that the equation has global smooth solutions for every reasonably nice initial data.

The assumption of Theorem \ref{t:main} holds for any power-law interaction $\alpha(r) = r^\gamma$ with $\gamma$ in the usual range $\gamma \in [-3,1]$. It even goes beyond that range if the Landau collision operator is understood appropriately taking into account the cancellations in the integrand of \eqref{e:landauoperator} for $|v-w|$ small. Our parameter $\sqrt{19}$ is not optimal. After we posted online a first version of the this article, Sehyun Ji improved the computation of the threshold of applicability of Theorem \ref{t:main} to $r |\alpha'(r)| / \alpha(r) \leq \sqrt{22}$ in \cite{sehyun2024}. There is no reason to expect this condition to be optimal either. We are currently not aware of any example, for any interaction potential, of a solution for which the Fisher information is not monotone decreasing.

As a consequence of the monotonicity of the Fisher information, we deduce the global existence of smooth solutions in the very-soft-potential range.
\begin{thm} \label{t:main2}
Assume $\alpha(r) = r^\gamma$, for $\gamma \in [-3,1]$. Let $f_0 : \R^3 \to [0,\infty)$ be an initial data that is bounded by a Maxwellian in the sense that
\[ f_0(v) \leq C_0 \exp(-\beta |v|^2), \]
for some positive parameters $C_0$ and $\beta$.

Then there is a unique global classical solution $f:[0,\infty)\times\mathbb{R}^3 \to [0,\infty)$ to the Landau equation \eqref{e:landauequation}, with initial data $f(0,v) = f_0(v)$. For any positive time, this function $f$ is strictly positive, in the Schwartz space, and bounded above by a Maxwellian. The Fisher information $i(f)$ is non-increasing. 
\end{thm}

The remarkable new feature of Theorem \ref{t:main2} is that it applies to the very-soft potential range $\gamma \leq -2$. The most important case of Theorem \ref{t:main2} is $\alpha(r) = r^{-3}$, which is the original equation by Landau and corresponds to the evolution of the velocity density of charged particles that interact by Coulomb potentials. Regularity estimates for this equation have remained an elusive well-known open problem for several years. The difficulty comes from the fact that the reaction term is too singular to be bested by the diffusion term when we only use the previously known controlled coercive quantities of mass, energy and entropy. The boundedness of the Fisher information, which is provided by Theorem \ref{t:main}, overcomes this difficulty altogether.

\begin{remark}
Theorem \ref{t:main2} is a relatively direct consequence of Theorem \ref{t:main} combined with well known techniques. There are various short-time well posedness results in the literature that one can apply, as well as conditional a-priori estimates. A slightly different version of Theorem \ref{t:main2} would result from different choices between them. For example, it is possible to replace the Maxwellian upper bound for $f_0$ with the condition $f_0 \in L^\infty_k$ for $k$ sufficiently large, and it would lead to correspondingly weaker decay conditions for the solution $f$.

Using our Theorem \ref{t:main}, minimalist assumptions on the initial data for which the Cauchy problem is solvable are investigated in the recent preprints \cite{amelie2024,sehyun2024dissipation}.
\end{remark}

\begin{remark}
A version of Theorem \ref{t:main} for the homogeneous Boltzmann equation is obtained in \cite{imbert2024monotonicity}. In that recent preprint, Cyril Imbert, Cedric Villani and the second author reproduce much of the analysis of this paper in the context of the Boltzmann equation. However, a new difficulty emerges that requires the development of a highly nontrivial integro-differential version of the inequality of Section \ref{s:logpoincare}.
\end{remark}

\subsection{Historical remarks and references}
Given that the Landau equation is one of the main equations in statistical mechanics, especially in the Coulombic case $\alpha(r) = r^{-3}$, it is no surprise that it has received a significant amount of attention from the mathematical community through the years. It is impossible to review all the literature in the subject. We discuss some of the most relevant results below.

The analysis of the Boltzmann equation, as well as the Landau equation, involves significantly simpler formulas in the case of the Maxwell-molecules model (which is $\alpha \equiv 1$) than in the case of general interaction potentials. One of the reasons is that there is a manageable expression for the Fourier transform of the operator due to Alexander Bobylev (see \cite{bobylev1975} and \cite{bobylev1988}). It is also well known that the evolution of the equation is contractive with respect to the quadratic Wasserstein distance (see \cite{tanaka1978} and \cite{bolley2007}), and that the Fisher information is non-increasing along the flow (see \cite{villani1998fisherboltzmann} and \cite{villani2000fisherlandau}). We naturally wonder if these properties fail to be true in the case of other interaction potentials, or if they hold in more generality but proofs are only well understood in the case of Maxwell molecules due to the simpler arithmetic structure of the operator. Is the Fisher information monotone decreasing for the Boltzmann and Landau equation for some interaction potential other than Maxwell molecules? What about the entropy dissipation? What about the contractivity of the Wasserstein distance? We answer the first of these questions for the Landau equation in Theorem \ref{t:main}.

It is rather unusual to find a new explicit Lyapunov functional for a very well studied equation in mathematical physics. Moreover, most basic conserved and monotone quantities in PDEs are verified by a relatively simple computation. It is also unusual to find a simple Lyapunov functional whose proof is nontrivial.

The first reference to the monotonicity of the Fisher information in a kinetic equation is in an interesting paper by McKean in 1966 \cite{mckean1966}. He considers Kac's 1D caricature of the Boltzmann equation with Maxwell molecules. He proves that the Fisher information is monotone in this context. The paper includes a number of opinions and conjectures including the following.
\begin{itemize}
\item He writes that the Fisher information ``\emph{probably fails to decrease}'' in the 3D problem.
\item He conjectures (in the context of the Kac equation) that the functionals that result as higher order derivatives of the entropy by heat flow are all monotone. He reports that he tried to prove it but could not do it.
\item He conjectures that the entropy dissipation for the Boltzmann equation should be monotone decreasing. He even suggests that its derivative may be monotone as well.
\end{itemize}

In 1992, Giuseppe Toscani proved the monotonicity of the Fisher information for the 2D Boltzmann equation in the Maxwell molecules case \cite{toscani1992}. Toscani conjectures that the Fisher information can probably be proved to be monotone ``\emph{at least in the case of inverse power interaction potentials}'', but he states that it is still an open question. 

In 1998, Cedric Villani proved the monotonicity of the Fisher information for the Boltzmann equation with Maxwell molecules in arbitrary dimension \cite{villani1998fisherboltzmann}. The proof, which one may say uses similar ideas as in Toscani's paper, involves more complicated geometry and formulas. In 2012, Matthes and Toscani provide a very short alternative proof based on Fourier analysis \cite{matthes2012}. This proof is very specific to $\alpha \equiv 1$ and provides no intuition about other interaction potentials. In the introduction of \cite{villani1998fisherboltzmann}, Villani writes ``\emph{We also investigate briefly the case of arbitrary potentials, and show precisely why the Maxwellian case seems to depart from the other}'', which reflects the limitations of methods available then when it came to general power law potentials.

The monotonicity of the Fisher information for the Boltzmann equation implies the monotonicity for the Landau equation as a limit case. In 2000, Villani wrote a direct proof in the case of the Landau equation with Maxwell molecules \cite{villani2000fisherlandau}. While the paper \cite{villani2000fisherlandau} is fairly short, the proof is nontrivial. The intuition is arguably harder to grasp than in the Boltzmann case. In Section \ref{s:mm}, we provide an alternative proof for the Landau equation in the Maxwell molecules case, which is in some ways the starting point for our general method for Theorem \ref{t:main}.

Other recent publications studying the evolution of the Fisher information are \cite{alonso2019} for the hard-potential case and \cite{meng2023} for moderately soft potentials. They present upper bounds (not monotonicity) for the Fisher information that are uniform in time. These are scenarios in which there are well known global-in-time regularity estimates.

\medskip

The great majority of the regularity estimates for the Landau equation in the past use a decomposition for the Landau collision operator \eqref{e:landauoperator} as a sum of a diffusion term plus a lower order term. In non-divergence form it reads
\[ q(f) = \bar a_{ij} \partial_{ij} f + \bar c f,\]
where
\[ \bar a_{ij} = \int_{\R^3} \alpha(|v-w|) a_{ij}(v-w) f(w),\]
and $\bar c = -\partial_{ij} \bar a_{ij}$. Some ellipticity bounds can be deduced for the coefficients $\bar a_{ij}$ based only on the mass, energy and entropy of $f$. The reaction term $\bar c f$ is more singular when $\gamma$ is more negative. The majority of the estimates in the literature are obtained using parabolic estimates for the diffusion term and using them to control the other term. We do not use this decomposition in our proof of Theorem \ref{t:main}. When $\gamma \geq 0$, the reaction term is very simple because $\bar c$ is bounded point-wise in terms of the mass and energy of $f$. When $\gamma \in [-2,0]$, the reaction term can still be controlled with the help of ellipticity estimates from the diffusion term. For the very-soft potential range $\gamma < -2$, the reaction term is too singular to be controlled from the diffusion term. This is vaguely the reason of the main difficulty in establishing unconditional bounds in the very-soft potential range.

The early results on classical well-posedness for the Landau equation focused on the Maxwell molecules case ($\gamma=0$) or hard potentials ($\gamma > 0$). Cedric Villani investigated the Maxwell-molecules case first \cite{villani1998spatiallyhomogeneous}, and later the case of hard potentials in collaboration with Laurent Desvillettes \cite{desvillettes2000,desvillattes2000II}. The regularity in the hard potentials case was revisited by El Safadi in \cite{elsafadi2007}. Our understanding of the global well posedness and smoothness for $\gamma  \geq 0$ is very satisfactory. For moderately soft potentials, there are regularity estimates in \cite{wu2014,silvestre2017,gualdani2019Ap}, from which one can construct global smooth solutions as in Theorem \ref{t:main2}.

For very soft potentials, which is the range $\gamma \in [-3,-2)$, the global classical well posedness of the equation has been an elusive and well-known open question for several years. This is the most interesting range because it includes the original Landau equation for Coulomb potentials which is $\gamma = -3$. The development of our understanding before this work was comparable with our current understanding of the Navier-Stokes equation (in terms of what the known results are). The results in the current literature can be roughly classified in the following groups.
\begin{description}
	\item[Global-in-time weak solutions.]

	In 1998, Cedric Villani defined a notion of generalized solution for the space-homogeneous Boltzmann and Landau equation \cite{villani1998Hsolutions}, which he called $H$-solution. He was able to prove the existence of global solutions of this kind, but not the uniqueness. In some way, Villani's result for the Landau equation plays the same role as Leray's global weak solutions do for the Navier-Stokes equation. See also \cite{peskov1977} for an earlier notion of generalized solution.


	From an entropy dissipation estimate, Laurent Desvillettes proved that the H-solutions constructed by Villani belong to $L^1([0,T], L^3_{-3}(\R^3))$ and are in fact weak solutions in a more classical sense \cite{desvillettes2015} (see also \cite{ji2023entropy}).

	Francois Golse, Maria Gualdani, Cyril Imbert, and Alexis Vasseur showed that the global weak solutions constructed by Villani are smooth outside of a potential set of times of dimension at most $1/2$ (see \cite{golse2022asens} and \cite{golse2021}).

	In a recent preprint \cite{golse2022local}, Francois Golse, Cyril Imbert and Alexis Vasseur prove that the set of potential singularities in $(t,v)$ has parabolic Hausdorff dimension at most $7/2$ in the case of Coulombic potential. This result is more or less comparable with the Caffarelli-Kohn-Nirenberg partial regularity theorem for Navier-Stokes, except that the dimension they obtain is larger in the case of the Landau equation. In \cite{golse2022local}, they also prove that axially symmetric solutions of the Landau equation with very soft potentials are smooth away from the axis of symmetry.

	\item[Short-time existence]

	Short time existence results are obtained by Nicolas Fournier and H\'el\`ene Gu\'erin \cite{fournier2009}, by Chris Henderson, Stanley Snelson and Andrei Tarfulea \cite{henderson2020}, by Hyung Ju Hwang and Jin Woo Jang \cite{hwang2020}, and by William Golding and Am\'elie Loher \cite{golding2023local}.

	\item[Perturbative theory around the Maxwellian]

	When the initial data is sufficiently close to a fixed Maxwellian in an appropriate sense, Yan Guo constructed a global smooth solution \cite{guo2002} (even for the space-inhomogeneous case). See also \cite{golding2023global} for a more recent result with weaker assumptions on the initial data but restricted to the space-homogeneous case.

	\item[Partial progress aiming at the regularity of solutions]

	Conditional regularity results are obtained by the second author in \cite{silvestre2017}, and by Maria Gualdani and the first author in \cite{gualdani2019Ap}. These works imply in particular that if the solution to \eqref{e:landauequation} stays bounded in certain $L^p$ space, for an appropriately large exponent $p$, then the solution will be smooth.

    Recent results by Ricardo Alonso, V\'eronique Bagland, Laurent Desvillettes and Bertrand Lods \cite{alonso2023solutions} and by William Golding and Am\'elie Loher \cite{golding2023local}, provide Prodi-Serrin type conditional regularity estimates depending on $f \in L^p([0,T],L^q(\R^3))$ for suitable pairs of exponents $p$ and $q$. 

	Other partial results regarding the regularity of the space-homogeneous Landau equation with Coulomb potentials are given in \cite{gamba2015,bedrossian2022,cabrera2023regularization,silvestre2023regularity,desvillettes2020new}.

	\item[Modified equations]

	In \cite{krieger2012}, Joachim Krieger and Robert Strain proposed an isotropic toy model that retains some of the features of the original Landau equation for the case of Coulomb potentials, and where blow up could be ruled out (for radial solutions). This model and other variations is further analyzed in \cite{gressman2012,gualdani2016,gualdani2022hardy,SnelsonIsotropic}.  Other simplified radially symmetric models are analyzed by Alexander Bobylev in \cite{bobylev2023}.
\end{description}

The difficulty of obtaining unconditional regularity estimates for the Landau equation with Coulomb potentials (resolved in Theorem \ref{t:main2}) is mentioned as an outstanding open problem in the majority of these papers.

The problem of establishing the regularity of the Landau equation with Coulomb potentials is also described in the open problems section of Villani's book \cite[Chapter 5. Section 1.3]{villani2002book}. Villani first argues that finite time blow up may be expected from an analogy between the Landau equation and the nonlinear heat equation. However, he reports that after seeing some numerical simulations by F. Filbet, he changed his mind and became convinced that blow-up should not occur. This is confirmed by our Theorem \ref{t:main2}.

In \cite{villani2025fisher}, Cedric Villani provides further details about the history of these problems, as well as a description of the ideas leading to the proof of the monotonicity of the Fisher information for the Boltzmann equation in \cite{imbert2024monotonicity}.

\subsection{Outline of the paper and main ideas in the proof.}
In this section we attempt to explain a rough outline of the ideas involved in the proof of Theorem \ref{t:main} and the general organization of the paper.

In Section \ref{s:preliminaries}, we review some of the results in the literature that we apply in this paper. None of the results presented in this preliminary section is new.

The first idea leading to Theorem \ref{t:main} is described in Section \ref{s:lifting}. We observe that we can write the Landau operator as a composition of a linear elliptic operator in $\R^6$, and the projection of a distribution function $F$ in $\R^6$ over its marginal. More precisely, for a function $F: \R^6 \to [0,\infty)$ with unit integral, we write
\[ \pi F (v) := \int_{\R^3} F(v,w) \dd w.\]
We also define the degenerate elliptic operator $Q$ acting on functions $F= F(v,w)$ as
\[ Q(F)(v,w) = (\partial_{v_i} - \partial_{w_i}) \left( \alpha(|v-w|) a_{ij}(v-w) (\partial_{v_j} - \partial_{w_j}) F \right). \]
With this notation, we observe that $q(f) = \pi Q(f \otimes f)$, where $(f\otimes f)(v,w) := f(v)f(w)$.

It turns out that the projection over marginals reduces the value of the Fisher information in the sense that for any symmetric probability distribution $F$ on $\R^6$, 
\[ i(f) = \frac 12 I(f\otimes f) \leq I(F) \ \text{ when } f = \pi F.\]
Here, we write $I(F)$ for the Fisher information applied to a function in $\R^6$. From this inequality we argue that the Fisher information will decrease along the direction of $q(f)$ if $I(F)$ decreases along the direction of $Q(F)$ for any function $F = F(v,w)$ defined on $\R^6$ such that $F(v,w) = F(w,v)$.

This is a substantial simplification. The Landau operator, $q$, which is nonlinear and nonlocal, is now replaced with a linear elliptic operator with explicit coefficients, $Q$. Within this framework the known monotonicity of the Fisher information for Maxwell molecules follows easily (see below). For all other potentials, estimating the derivative of $I$ along $Q$ is still nontrivial and requires further ideas.

In Section \ref{s:vectorsb} we observe that for certain vector fields $\tilde b_1$, $\tilde b_2$ and $\tilde b_3$ we have
\[ Q(F) = \sum_{k=1,2,3} \sqrt \alpha \ \tilde b_k \cdot \nabla \left( \sqrt \alpha \ \tilde b_k \cdot \nabla F \right).\]
This observation, while elementary, will help us organize our formulas efficiently. We use these vectors and their properties extensively in the rest of the paper.

The transport of each of these vector fields $\tilde b_k \cdot \nabla$ individually goes by isometries in $\R^6$. Thus, $I(F)$ is constant along the flow of $\tilde b_k \cdot \nabla$. The fact that $I(F)$ is monotone decreasing in the direction of $Q(F)$ in the case of Maxwell molecules follows immediately from this and the convexity of the Fisher information $I$. This is explained in Section \ref{s:mm}.

Whenever $\alpha$ is not constant, the flow of $\sqrt \alpha \ \tilde b_k \cdot \nabla$ is not an isometry in $\R^6$. However, it is an isometry when we restrict it to the level sets of $\alpha$. With this idea in mind, in Section \ref{s:alternatives}, we define a modified Fisher information $I_{tan}$ that only takes into account the gradient of $F$ along the directions parallel to the level sets of $|v-w|$. This functional $I_{tan}$ is constant along the flow of $\sqrt \alpha \ \tilde b_k \cdot \nabla$. Using the same convexity argument as in the case of Maxwell molecules, we conclude that $\langle I_{tan}'(F),Q(F) \rangle \leq 0$, with a precise estimate of its dissipation in terms of second derivatives of $F$.

The difference between the full Fisher information $I(F)$ and the modified one $I_{tan}(F)$ is only one direction, which we call $\n$, and is the normalized gradient of $|v-w|$ in $\R^6$. We write 
\[ J(F) := I(F) - I_{tan}(F) = \iint_{\R^6} \frac{|\n \cdot \nabla F|^2} F \dd w \dd v.\]
In Section \ref{s:liebrackets}, we compute $\langle J'(F), Q(F) \rangle$ to express it in a form that resembles as closely as possible the structure of the dissipation of $I_{tan}$. We do it through a direct computation using the vector fields $\tilde b_k$ and Lie brackets. We end up with a negative diffusion term involving second derivatives of $F$ and a positive error term of the form
\[ R := \iint_{\R^6} \frac{|\alpha'|^2}{\alpha} \left( \tilde b_k \cdot \nabla \log F \right)^2 F \dd w \dd v.\]

In Section \ref{s:threediffusionterms}, we break the favorable diffusion terms obtained in the previous sections into three groups, depending on the directions of their second derivatives. We isolate one term, which we call $D_{spherical}$, that only takes into account second derivatives in the directions of the vector fields $\tilde b_k$'s. We claim that this diffusion term can be used to control the error term $R$. After using polar coordinates and performing some manipulations, we reduce the problem to an inequality for functions on the sphere $f : S^2 \to (0,\infty)$ of the following form
\[ \int_{S^2} |\nabla^2_\sigma \log f|^2 f \dd \sigma \geq \lambda \int_{S^2} |\nabla_\sigma \log f|^2 f \dd \sigma. \]
It is not difficult to prove that the inequality holds for some $\lambda>0$. However, it is important for us to obtain a nearly sharp constant $\lambda$. The maximum value that we are able to handle for $r |\alpha'(r)|/\alpha(r)$ in Theorem \ref{t:main} (whose value is $\sqrt{19}$) depends on the value of $\lambda$ from this inequality. We prove the inequality in Section \ref{s:logpoincare} (written slightly differently in terms of the vector fields $b_k$). It appears to be new as far as we are aware.

Finally, in Section \ref{s:globalexistence}, we explain how Theorem \ref{t:main}, combined with results from the literature, can be used to derive Theorem \ref{t:main2}. There is no difficulty in Section \ref{s:globalexistence}.

\subsection*{Acknowledgements}

We would like to thank David Bowman and Sehyun Ji for reading an earlier version of this manuscript and pointing out some typos. We also thank Maria Gualdani and Cyril Imbert for many useful discussions about this problem through several years. {We would also like to thank the anonymous reviewer for many helpful comments and suggestions.

\section{Preliminaries}
\label{s:preliminaries}

In this section we collect the results from the literature that we use in this paper. The results we state here are well known. They either appear explicitly or are straight forward minor modifications of the results in the literature. We provide references in every case.

As it is customary, we write $\langle v \rangle = (1+|v|^2)^{1/2}$ and, for any $p \in [1,\infty]$, $k \in \R$,
\[ \|f\|_{L^p_k(\R^3)} := \|\langle v \rangle^k f \|_{L^p(\R^3)}.\]

The mass, momentum and energy are constant along the flow of the equation. That is
\begin{align*}
\text{(mass)} & := \int_{\R^3} f(v) \dd v, \\
\text{(momentum)} & := \int_{\R^3} f(v) v \dd v, \\
\text{(energy)} & := \int_{\R^3} f(v) |v|^2 \dd v.
\end{align*}

To fix ideas, it is a comfortable choice to study solutions of the equation \eqref{e:landauequation} with unit mass and zero momentum. This will be preserved along the flow of the equation.

We recall the formula for the entropy.
\[ h(t) := \int_{\R^3} f(v) \log f(v) \dd v. \]
The entropy is non-increasing for solutions of \eqref{e:landauequation}. Before this work, it was the only natural quantity associated to the equation known to be monotone decreasing -- outside the Maxwell-molecules case $\alpha \equiv 1$.

To begin the analysis of the equation, we start from a short-time existence result. There exist a few in the literature, both in the space-homogeneous and space-inhomogeneous regime. The most complete that we know of is given in \cite{henderson2020}. Restricted to the space-homogeneous regime, it becomes the following result.

\begin{thm}[from \cite{henderson2020}] \label{t:henderson-snelson-tarfulea}
Assume $\alpha(r) = r^\gamma$ for some $\gamma \in [-3,0)$. Let $f_0: \R^3 \to [0,\infty)$ be an initial value so that $f_0 \in L^\infty_{k_0}$ with $k_0 = \max(5,15/(5+\gamma))$. There exists a positive time $T>0$ and a classical solution $f : [0,T] \times \R^3 \to [0,\infty)$ to the equation \eqref{e:landauequation} with initial data $f_0$. This solution satisfies the following properties.
\begin{itemize}
    \item $f(t,v) > 0$ for all $t>0$ and $v \in \R^3$.
    \item $f \in C^\infty((0,T)\times \R^3)$ with upper bounds for $t>0$ for all derivatives.
    \item If $f_0$ is continuous, then $f$ matches the initial data continuously. Otherwise, it is weakly continuous at $t=0$.
    \item The solution can be extended for as long as $\|f(t)\|_{L^{\infty}_{k_0}} < +\infty$.
    \item Let $T_E$ be the maximum time of existence and $\tau>0$. For all $m \in \mathbb N$, there exists a $k_m$ and an $\ell_m$ so that if $f_0 \in L^{\infty}_{k_m}$ then $D^{(m)} f(t) \in L^\infty([0,T_E-\tau],L^{\infty}_{k_m-\ell_m}(\R^3))$.

\end{itemize}
\end{thm}

The main result in \cite{henderson2020} is for the space-inhomogeneous equation. There is an extra assumption for the initial data $f_0$, which they call \emph{well distributed}. This assumption is only relevant for the space-inhomogeneous case. In the space-homogeneous case, it would be automatically satisfied if $f_0$ is continuous and nonzero. Even if $f_0$ is not continuous, it is not difficult to see that $f(t,v)$ will be well distributed for any $t>0$ small.

There is also a more recent short-time existence result \cite{golding2023local}, which is specific to the space-homogeneous Landau equation with Coulomb potentials, for initial data that is merely in $L^1_k \cap L^p$ for some $p > 3/2$. The solution constructed in \cite{golding2023local} is also classical and smooth. The continuation criteria presented in \cite{golding2023local} is a simpler Prodi-Serrin type condition without weights.

The following result can be found in \cite[Theorem 3.7]{silvestre2017}. We use it to improve the continuation criteria from Theorem \ref{t:henderson-snelson-tarfulea} in the space-homogeneous case.

\begin{thm}[from \cite{silvestre2017}] \label{t:silvestre2017}
Assume $\alpha(r) = r^\gamma$ for some $\gamma \in [-3,2]$. Let $f : [0,T] \times \R^3 \to [0,\infty)$ be a classical solution to \eqref{e:landauequation}. Let $p>3/(5+\gamma)$, $p \geq 1$ and $k$ be sufficiently large (depending on $\gamma$). Assume that for all $t \in [0,T]$,
\[ \|f(t)\|_{L^p_k} \leq C_0.\]
Then $f$ is bounded for positive time with an upper bound
\[ \|f(t)\|_{L^\infty} \leq C_1(1+t^{-3/(2p)}) \]
for a constant $C_1$ that depends on $C_0$, the mass of $f_0$. and $\delta$.
\end{thm}

While the conditional upper bound in Theorem \ref{t:silvestre2017} does not readily provide any decay for large velocities, they can be deduced using the technique in \cite{cameron2018}. The results in \cite{cameron2018} are only stated in the moderately-soft potential range $\gamma \in [-2,0]$. However, they apply in the full range of soft potentials when we know in addition that the solution $f$ is bounded.

\begin{thm}[essentially from \cite{cameron2018}] \label{t:cameron2018}
Assume $\alpha(r) = r^\gamma$ for some $\gamma \in [-3,0]$. Let $f : [0,T] \times \R^3 \to [0,\infty)$ be a classical solution to \eqref{e:landauequation}. Assume that for $k$ sufficiently large and all $t \in [0,T]$,
\[ \|f(t)\|_{L^\infty} \leq C_0, \qquad \|f(t)\|_{L^1_k} \leq C_0.\]
Assume further that $f_0 \leq C_1 \exp(-\beta |v|^2)$ for some $C_1>0$ and $\beta>0$ sufficiently small. Then
\[ f(t,v) \leq C_2 \exp(-\beta |v|^2).\]
for a constant $C_2$ that depends on $C_0$, $C_1$ and the mass of $f_0$.
\end{thm}

\begin{proof}[Sketch proof]
Theorem \ref{t:cameron2018} is proved in \cite[Section 5]{cameron2018} for the moderately soft potential range $\gamma \in (-2,0]$. The fact $\gamma > -2$ is only used for the initial $L^\infty$ bounds and the upper bounds on the coefficients given in \cite[Lemma 2.1]{cameron2018}. We included the extra assumptions that $\|f\|_{L^\infty} \leq C_0$ and $\|f(t)\|_{L^1_k} \leq C_0$ to take care of those requirements. We are only left to check that the upper bounds on the coefficients can be verified with our extra assumptions.

For $\gamma \leq -2$, following the proof of the computation for upper bounds for the coefficients $\bar a_{ij}$ in \cite[Lemma 2.1]{cameron2018}, we get
\begin{align*}
\bar a_{ij} e_i e_j &\lesssim \int |v-w|^{\gamma+2} f(w) \dd w \\
&\leq \int_{B_r(v)} |v-w|^{\gamma+2} f(w) \dd w + \int_{B_{|v|/2} \setminus B_r(v)} |v-w|^{\gamma+2} f(w) \dd w + \int_{\R^3 \setminus B_{|v|/2} \setminus B_r(v)} |v-w|^{\gamma+2} f(w) \dd w \\
&\lesssim r^{\gamma+5} \|f\|_{L^\infty} + \langle v \rangle^{\gamma+2} \|f\|_{L^1} + r^{\gamma+2} \langle v \rangle^{-k} \|f\|_{L^1_k}.
\end{align*}
We choose $r = \langle v \rangle^{(\gamma+2)/(\gamma+5)}$ and then $k$ large enough so that $r^{\gamma+2} \langle v \rangle^{-k} \leq \langle v \rangle^{\gamma+2}$. Thus, we obtain the same upper bounds for the coefficients $\bar a_{ij}$ as in \cite[Lemma 2.1]{cameron2018} for the case $\gamma \in [-3,-2]$ when we assume in addition that $\|f\|_{L^\infty}$ and $\|f\|_{L^1_k}$ are bounded. 

The upper bound in \cite[Lemma 2.1]{cameron2018} improves when $e$ is parallel to $v$, which is proved along the same lines. Once these upper bounds are established, the rest of Section 5 in \cite{cameron2018} follows without any major change.
\end{proof}

We also recall a result on the propagation of bounded moments. The following theorem is a simplified version of a result in \cite{desvillettes2015}.

\begin{thm} \label{t:moments}
If $f$ is a solution of the Landau equation \eqref{e:landauequation}, where $\alpha(r) = r^\gamma$ with $\gamma \in [-3,1]$. Assume that the initial data $f_0$ belongs to $L^1_k$ for some exponent $k \geq 0$. Then for all $t \in [0,T]$, $\|f(t)\|_{L^1_k} \leq C(T)$, for some upper bound $C(T)$ depending on $T$, $\|f_0\|_{L^1_k}$, and the mass, energy and entropy of $f_0$
\end{thm}

Theorem \ref{t:moments} is stated in \cite{desvillettes2015} for the Coulomb case $\gamma=-3$ only. The proof easily applies to any soft potential in the range $\gamma \in [-3,0]$ (and even slightly smaller than $-3$). The result in \cite{desvillettes2015} is also stated and proved for H-solutions, which makes it technically more complicated. The case $\gamma=0$ had been covered earlier in \cite{villani1998spatiallyhomogeneous}. For $\gamma>0$, it is proved in \cite{desvillettes2000} that there is even a gain of moments in the sense that $\|f(t)\|_{L^1_k}$ is bounded for $t>0$ even if $\|f_0\|_{L^1_k}$ is not. Note that for $\gamma<0$, even though $C(T)$ may a-priori grow for large time, it does not blow up in finite time.

For uniqueness of solutions, we cite \cite{fournier2009} for $\gamma >-3$ and \cite{fournier2010uniqueness} for $\gamma = -3$. The following theorem summarizes both cases.
\begin{thm} \label{t:uniqueness}
When $\alpha(r) = r^\gamma$ with $\gamma \in [-3,2)$, there exists at most one weak solution of \eqref{e:landauequation} that belongs to the space
\[ f \in L^\infty([0,T],L^1_2(\R^3)) \cap L^1([0,T],L^p(\R^3)).\]
Here $p=\infty$ for $\gamma=-3$, and $p=3/(3+\gamma)$ otherwise.
\end{thm}

The solutions that we obtain from Theorem \ref{t:henderson-snelson-tarfulea} are $C^\infty$ smooth for $t>0$. If $f_0$ is not continuous, the initial data is only achieved in the sense of weak continuity. Because of that, it is not irrelevant that Theorem \ref{t:uniqueness} applies to weak solutions. It tells us in particular that there cannot be two different solutions bifurcating from the same initial data at time zero. The precise definition of weak solution given in \cite{fournier2010uniqueness,fournier2009} is not necessarily important here, but only the fact that it is compatible with a smooth solution that achieves a potentially rough initial data by weak continuity.

As we mentioned in the introduction, the Fisher information of a nonnegative function $f:\R^3 \to [0,\infty)$ is given by the formulas
\begin{align*}
i(f) &:= \int_{\R^3} \frac{|\nabla f|^2}f \dd v \\
&= \int_{\R^3} |\nabla \log f|^2 \ f \dd v \\
&= 4 \int_{\R^3} |\nabla \sqrt f|^2 \dd v.
\end{align*}

Something must be clarified about the case when $f$ has vacuum regions. In that case the third formula is the only one that makes immediate sense and suggests that we should set the integrand as zero wherever $f=0$. This apparent difficulty with vacuum regions is actually not relevant for our analysis. In \cite{henderson2020}, the authors prove that the solution to the Landau equation becomes immediately strictly positive (even in the space inhomogeneous setting).

The Fisher information $i(f)$ is well defined for nonnegative functions $f$ so that $\sqrt f \in \dot H^1(\R^3)$. This will always be the case with the class of solutions that we work with (see Appendix \ref{a:byparts}). Moreover, the following lemma from \cite{toscani2000} indicates that the Fisher information is finite for all nonnegative functions that are sufficiently smooth and fast decaying at infinity.
\begin{lemma} \label{l:finite-fisher}
  For all $\eps >0$, there exists $C_\eps >0$ (only depending on dimension and $\eps$) such that for all $f \in H^2_{3/2+\eps}(\R^3)$,
  \[ I (f) \leq C_\eps \| f\|_{H^2_{3/2+\eps}}.\]
\end{lemma}

We use the notation $\langle i'(f),g \rangle$ to denote the Gateaux derivative of $i$ in the direction $g$. More precisely, if $f(t)$ is a curve of non-negative functions so that $\sqrt{f(t)} \in \dot H^1$ and $\partial_t f(0) = g$, then by definition
\[ \langle i'(f), g \rangle = \partial_t i(f(t))_{|t = 0}.\]
Likewise, $i''(f)$ denotes the quadratic form so that if $f(t)$ is second differentiable with respect to $t$ then,
\[ \langle i''(f)g, g \rangle + \langle i'(f), \partial_{tt} f_{|t=0} \rangle = \partial_{tt} i(f(t))_{|t = 0}.\]
It is a standard fact that the values of these formulas do not depend on the curve $f(t)$ other than through the values of $g=\partial_t f(0)$. Moreover, the convexity of $i$ corresponds to $i''(f)$ being a positive quadratic form.

We use Theorem \ref{t:henderson-snelson-tarfulea} as the starting point to construct our global smooth solution in Theorem \ref{t:main2}. Combining it with the propagation of moments in Theorem \ref{t:moments} and the Maxwellian upper bounds of Theorem \ref{t:cameron2018}, the solutions we work with are very well behaved functions. They are $C^\infty$ smooth functions, they have rapid decay at infinity together with all their derivatives, and they are strictly positive. This is the class of functions that we work with in this paper, which minimizes the technical inconveniences in our computations. A posteriori, Theorem \ref{t:main}, and most of the lemmas in this paper, apply to a much wider class of functions by a density argument (see Appendix \ref{a:byparts}). However, we will make no attempt to classify the most general class of functions for which our results make sense.

Even when $f$ and $|\nabla f|$ decay faster than any algebraic rate as $|v| \to \infty$, it is not completely obvious that $|\nabla f|^2 / f$ is integrable in $\R^3$. Moreover, there are several integrals throughout this paper that involve first, second, and even third derivatives of $\log f$, and these quantities may grow for large velocities. In the same vein as Lemma \ref{l:finite-fisher}, it can be argued that if sufficiently many derivatives of $f$ decay sufficiently fast as $|v| \to \infty$, then all the integrals in this article make sense. However, we can also avoid this technical annoyance altogether by approximating the function $f$ with one that is bounded from below and above by a multiple of the same Maxwellian. We do not want to distract the reader with this very standard technical point, so we explain it in Appendix \ref{a:byparts}. For the rest of this paper, we manipulate integrals involving the function $f$ (or later $F$ in the lifted variables) multiplying various derivatives of its logarithm without further comments.

\section{Lifting and projection}
\label{s:lifting}

We have found that the question of monotonicity of $i(f)$ for $f$ solving \eqref{e:landauequation} is essentially the same as the question of monotonicity of the Fisher information for a specific linear second order equation in a space with double the original number of variables. This is a significant simplification, it means we can understand our question by studying a simpler and concrete evolution equation, and one that is linear and local -- in contrast with the original equation. The starting point is realizing that $q(f)$ is defined as the integral of a function in a space with twice the number of variables, and therefore it makes sense to ``lift'' the flow and any quantity of interest into this space with double the variables. The purpose of this section is to describe the usage of these lifting and projection operations.

Given a function $f :\R^d \to \R$, we can build a function $F:\R^{d+d} \to \R$ by taking the tensor product of $f$ with itself, as follows
\begin{align*}
  F(v,w) = (f \otimes f)(v,w) := f(v) f(w).
\end{align*}
This defines a map from scalar functions in $\mathbb{R}^d$ to scalar functions in $\mathbb{R}^{d+d}$. Conversely, we define the following \emph{projection} operator 
\[ \pi(F)(v) := \int_{\R^d} F(v,w) \dd w .\]
This is an operator that maps a function $F(v,w)$ defined in $\R^{d+d}$ into a function $\pi(F) : \R^d \to \R$. The map $\pi$ is a bounded linear map from $L^1(\mathbb{R}^{d+d})$ to $L^1(\mathbb{R}^d)$. Furthermore observe that if $F=F(v,w)$ is a probability density in $\mathbb{R}^{d+d}$, then $\pi(F)$ is simply the marginal probability distribution in the $v \in \mathbb{R}^d$ variable. 

We will loosely use the terminology that the variables $(v,w) \in \mathbb{R}^{d+d}$ are the ``lifted variables'', functions in $\mathbb{R}^{d+d}$ are lifted functions, and so on. The next ingredient is a second order differential operator for lifted functions. Concretely, we consider the following degenerate-elliptic operator that applies to functions $F = F(v,w)$. 
\begin{equation} \label{e:Q}
 Q(F)(v,w) := (\partial_{v_i} - \partial_{w_i}) \left( \alpha(|v-w|) a_{ij}(v-w) \left( \partial_{v_j} - \partial_{w_j} \right) F(v,w) \right). 
\end{equation}
For any function $f = f(v)$, we may study the initial value problem associated to the linear operator $Q$ with initial data $f \otimes f$.
\begin{equation} \label{e:liftedevolution}
\begin{aligned}
F(0,v,w) &= f(v) f(w), \\
F_t &= Q(F)
\end{aligned}
\end{equation}

The equation \eqref{e:liftedevolution} is a linear degenerate parabolic equation with explicit coefficients. Its classical well posedness is straight-forward from classical PDE theory.\footnote{In fact, one can see that the equation \eqref{e:liftedevolution} is literally the heat equation on each of the spheres that we describe in \eqref{e:sphere} with diffusion coefficient $\alpha$. See Remark \ref{r:sphere}.} The integral of $F$ on $\R^6$ is constant in time. From the symmetry of the initial data, and the symmetry of the equation, we deduce that for all $t>0$ the solution must be symmetric in the sense that $F(t,v,w) = F(t,w,v)$.

\begin{lemma} \label{l:qwithpi}
For any twice-differentiable function $f$ which is non-negative and has mass one, the Landau operator $q$ (given in \eqref{e:landauoperator}) coincides with
\[ q(f) = \pi (Q(F)) = \partial_t \left[ \pi F \right]_{|t=0},\]
where $F$ is the solution of \eqref{e:liftedevolution}.
\end{lemma}

\begin{proof} The proof follows by a straight-forward substitution. We see that by the linearity of $\pi$
\begin{align*}
\partial_t \left[ \pi (F) \right]_{|t=0} &= \pi (Q(F))_{|t =0} \\
&= \int_{\R^3} \left( (\partial_{v_i} - \partial_{w_i}) \left( \alpha(|v-w|) a_{ij}(v-w) \left( \partial_{v_j} - \partial_{w_j} \right) F(0,v,w) \right) \right) \dd w
\intertext{The derivatives with respect to $w_i$ integrate to zero respect to $\dd w$. Thus,}
&= \partial_{v_i} \int_{\R^3} \alpha(|v-w|) a_{ij}(v-w) \left( \partial_{v_j} - \partial_{w_j} \right) F(0,v,w) \dd w \\
&= q(f).
\end{align*}
\end{proof}
We use the letter $I$ for the Fisher information in the \emph{lifted} variables. For $F : \R^{6} \to (0,\infty)$,
\begin{equation}  \label{e:fisherlifted}
\begin{aligned}
I(F) &:= \iint_{\R^{6}} \frac{|\nabla F|^2}F \dd w \dd v \\
&= \iint_{\R^{6}} |\nabla \log F|^2 \ F \dd w \dd v \\
&= 4 \iint_{\R^{6}} |\nabla \sqrt F|^2 \dd w \dd v
\end{aligned}
\end{equation}
When $F$ has vacuum regions, only the third line makes sense literally. For the first two formulas to make sense, and agree with the third, we must set $|\nabla F|^2 / F = |\nabla \log F|^2 F = 0$ when $F=0$.

The next lemmas relate $i(f)$ with $I(F)$. The first one is very standard.
\begin{lemma} \label{l:i(f)=I(fxf)}
 For any non-negative $C^1$ function $f$ with $\int f\;dv =1$ and $i(f)<\infty$, we have
\[ i(f) = \frac 12 I(f \otimes f).\]
\end{lemma}

\begin{proof}
We are going to show a bit more: if $f$ is a $C^1$ function with $f\geq 0$ and $f$ is not identically zero, then
\begin{align*}
  I(f\otimes f) = 2\left ( \int_{\mathbb{R}^3} f\dd v\right ) i(f).  
\end{align*}
First, let us prove this when $f(v)>0$ for all $v$. Since $(f\otimes f)(v,w) = f(v)f(w)$,
\begin{align*}
  \nabla_{\mathbb{R}^{6}}(f\otimes f)(v,w) = (f(w)\nabla_{\mathbb{R}^3} f(v),f(v)\nabla_{\mathbb{R}^3} f(w)),
\end{align*}
so, 
\begin{align*}
  |\nabla (f\otimes f)|^2 = f(w)^2|\nabla f(v)|^2+f(v)^2|\nabla f(w)|^2.
\end{align*}
In particular, 
\begin{align*}
  \frac{|\nabla (f\otimes f)|^2}{(f\otimes f)}(v,w) & = \frac{f(w)^2|\nabla f(v)|^2+f(v)^2|\nabla f(w)|^2}{f(v)f(w)}\\
    & = f(w)\frac{|\nabla f(v)|^2}{f(v)}+f(v)\frac{|\nabla f(w)|^2}{f(w)}.
\end{align*}		
Then, we integrate over $\mathbb{R}^{6}$,
\begin{align*}
  I(f\otimes f) & = \iint_{\R^{6}}\frac{|\nabla (f\otimes f)|^2}{(f\otimes f)}(v,w)\dd w \dd v \\ 
    & = \int_{\R^3}\int_{\R^3} f(w)\frac{|\nabla f(v)|^2}{f(v)}+f(v)\frac{|\nabla f(w)|^2}{f(w)}\dd w \dd v\\
	& = 2\left (\int_{\R^3} f(w)\dd w \right )\left ( \int_{\R^3} \frac{|\nabla f(v)|^2}{f(v)}\dd v\right ).
\end{align*}		
For a general non-negative $f$ of class $C^1$, we write $f_\varepsilon = f + \varepsilon \phi$, where $\phi(x) = e^{-|v|^2}$. Then $f_\varepsilon>0$ for every $\varepsilon>0$ and
\begin{align*}
  I(f_\varepsilon\otimes f_\varepsilon) = 2\left ( \int_{\mathbb{R}^3} f_\varepsilon\;dv\right ) i(f_\varepsilon).  
\end{align*}
Now we note that
\begin{align*}
  i(f_\varepsilon) & = \int_{\mathbb{R}^3} \frac{|\nabla f(v)|^2+2 \varepsilon \nabla f(v)\cdot \nabla \phi(v) + \varepsilon^2 |\nabla \phi(v)|^2}{f_\varepsilon(v)}\;dv\\
    & = \int_{\mathbb{R}^3} \frac{|\nabla f(v)|^2}{f_\varepsilon(v)}\dd v + \int_{\mathbb{R}^3}\frac{2\varepsilon \nabla f(v)\cdot \nabla \phi(v) + \varepsilon^2 |\nabla \phi(v)|^2}{f_\varepsilon(v)}\dd v
\end{align*}
The first integral converges to $i(f)$ as $\varepsilon \to 0$ thanks to $i(f)<\infty$ and the dominated convergence theorem. The absolute value of the second integral is bounded from above by
\begin{align*}
  \varepsilon \int_{\R^3}\frac{|\nabla f(v)|^2}{f(v)}\;dv + \varepsilon(\varepsilon+1)\int_{\R^3} \frac{|\nabla \phi(v)|^2}{\phi(v)}\;dv.
\end{align*}
It follows that $i(f_\varepsilon) \to i(f)$. It is obvious that $\int_{\mathbb{R}^3} f_\varepsilon(v)\;dv \to \int_{\mathbb{R}^3} f(v)\;dv$. Lastly,
\begin{align*}
  f_\varepsilon \otimes f_\varepsilon = f\otimes f + \varepsilon (f \otimes \phi + \phi \otimes f) + \varepsilon^2 \phi\otimes \phi.
\end{align*}
From here one can see that $I(f_\varepsilon\otimes f_\varepsilon) \to I(f\otimes f)$ as $\varepsilon \to 0$ in a manner entirely analogous to $i(f_\varepsilon) \to i(f)$, by noting that $f_\varepsilon \otimes f_\varepsilon \geq \max\{f\otimes f,\phi\otimes f\}$ for every $\varepsilon > 0$ and $\nabla_{\mathbb{R}^{6}}(f_\varepsilon \otimes f_\varepsilon) \to \nabla_{\mathbb{R}^3}(f\otimes f)$ pointwise as $\varepsilon \to 0$. 
\end{proof}{

The following is a crucial property of the Fisher information.
\begin{lemma} \label{l:I(F)>2i(pi(F))}
 For any non-negative $C^1$ function $F:\R^{6}\to [0,\infty)$ with $F(v,w)=F(w,v)$, we have
\[ i(\pi F) \leq \frac 12 I(F).\]
\end{lemma}

The inequality in Lemma \ref{l:I(F)>2i(pi(F))} was first observed by Eric Carlen in \cite{carlen1991}, without the symmetry assumption. There, he showed that some classical results related to the Fisher information follow as a consequence of this elementary inequality.

\begin{proof}
  First we show inequality under the additional assumption that $F>0$ everywhere. We note that
  \begin{align*}
    \nabla (\pi F)(v) & = \nabla_v \left \{\int_{\mathbb{R}^3}F(v,w)\dd w \right \}\\
	  & = \int_{\mathbb{R}^3}\nabla_v F(v,w)\dd w,
	  \intertext{so }
	|\nabla (\pi F)(v)| & \leq \int_{\mathbb{R}^3}|\nabla_v F(v,w)|\dd w.
  \end{align*}
  Since $F>0$ everywhere, we may apply the Cauchy-Schwarz inequality to get
  \begin{align*}
    \left (\int_{\mathbb{R}^3}|\nabla_v F(v,w)|\dd w \right)^2 \leq \left (\int_{\mathbb{R}^3}\frac{|\nabla_v F(v,w)|^2}{F(v,w)}\dd w\right ) \left ( \int_{\mathbb{R}^3}F(v,w)\dd w\right ).
  \end{align*}
  The last two inequalities combine to
  \begin{align*}
    |\nabla (\pi F)(v)|^2 \leq \left (\int_{\mathbb{R}^3}\frac{|\nabla_v F(v,w)|^2}{F(v,w)}\dd w\right )(\pi F)(v),\;\forall\;v.
  \end{align*}
  Dividing by $(\pi F)(v)$ (which is $>0$ for all $v$ since $F>0$ everywhere) and integrating in $v$, 
  \begin{align*}
    i(\pi F) \leq \iint_{\R^3 \times \mathbb{R}^3}\frac{|\nabla_v F(v,w)|^2}{F(v,w)}\dd w \dd v = \frac{1}{2}I(F).
  \end{align*}
  To obtain the last equality we have used that $F(v,w) = F(w,v)$. 
  
  Now let $F$ be just as in the statement of the lemma, and for every $\varepsilon>0$ define
  \begin{align*}
    F_\varepsilon(v,w) = F(v,w) + \varepsilon e^{-|v|^2-|w|^2}.  	  
  \end{align*}	  
  Then $F_\varepsilon$ also satisfies all the assumptions of the lemma, but in addition $F_\varepsilon>0$ everywhere, and thus
  \begin{align*}
     i(\pi F_\varepsilon) \leq \frac{1}{2}I(F_\varepsilon),\;\forall\;\varepsilon>0.
  \end{align*}
  Observe that $(\pi F_\varepsilon )(v)= (\pi F)(v) + c\varepsilon e^{-|v|^2} $ for some dimensional constant $c$. By repeating the argument in the proof of the previous lemma we conclude that $i(\pi F_\varepsilon) \to i(\pi F)$ and $I(F_\varepsilon) \to I(F)$ as $\varepsilon \to 0^+$, and the lemma follows. 
\end{proof}

Lemma \ref{l:I(F)>2i(pi(F))} is very important to make this program work. It would also hold if we replace the Fisher information with the usual entropy, but it would not hold for most usual quantities related to $f$ that involve derivatives.

\begin{lemma}\label{l:dti=dtI/2}
Let $f$ be a solution of \eqref{e:landauequation} and $F$ as in \eqref{e:liftedevolution}. Then 
\[ \partial_t i(f) \leq \frac 12 \partial_t I(F)_{|t = 0}. \]
\end{lemma}

\begin{proof}
Using Lemma \ref{l:qwithpi}, we get
\begin{align*}
\partial_t i(f) &= \langle i'(f), q(f) \rangle \\
&= \langle i'(f), \partial_t \left[ \pi F \right]_{|t=0} \rangle \\
&= \partial_t i(\pi F)_{|t=0} \\
&= \partial_t \left[ i(\pi F) - \frac 12 I(F) \right]_{|t=0} + \frac 12 \partial_t I(F)_{|t=0}.
\end{align*}
Because of Lemma \ref{l:I(F)>2i(pi(F))}, the first term is the time derivative of a nonpositive function that achieves its maximum at $t=0$. Therefore, it cannot be positive. We end up with
\[ \partial_t i(f) \leq \frac 12 \partial_t I(F)_{|t=0}.\]
\end{proof}

It is convenient to rewrite Lemma \ref{l:dti=dtI/2} in terms of the Gateaux derivatives of $i$ and $I$. It says that for $F = f \otimes f$, we have
\begin{equation}  \label{e:fromitoI}
\langle i'(f), q(f) \rangle \leq \frac 12 \langle I'(F), Q(F) \rangle.
\end{equation}
In order to prove that the left-hand side is negative, we will show that the right hand side is. The remarkable advantage of our analysis is that the question of monotonicity of the Fisher information for the Landau equation, which is nonlinear and nonlocal, is now reduced to its monotonicity for a linear equation with explicit coefficients, in the doubled-variables. Precisely, we will prove Theorem \ref{t:main} as soon as we prove the following inequality for any function $F : \R^6 \to [0,\infty)$ so that $F(v,w) = F(w,v)$.
\begin{equation} \label{e:question}
\langle I'(F), Q(F) \rangle \leq 0.
\end{equation}

Proving \eqref{e:question} is the objective of the coming sections. We want to apply it to solutions of the Landau equation. From Theorem \ref{t:henderson-snelson-tarfulea}, these solutions are strictly positive, smooth and rapidly decaying functions. We must show that the inequality \eqref{e:question} holds for any smooth function $F :\R^6 \to (0,\infty)$ with sufficient decay at infinity. There is no time evolution explicitly involved in the inequality \eqref{e:question}.

\begin{remark}
The inequality \eqref{e:fromitoI} is in fact an equality. There is a very elementary justification that we explain here. Consider the function $F(t,v,w)$ given by
\[ F(t,v,w) = f(v) f(w) + t Q(f \otimes f).\]
Suppose that $F(t,v,w) \geq 0$ at least in some small interval of time $t \in (-\delta,\delta)$. Otherwise, we may approximate it with a nonnegative function appropriately.

Let $\iota(t)$ be the function
\[ \iota(t) := i(\pi F(t,\cdot)) - \frac 12 I(F(t,\cdot)).\]
Since $F(0,v,w) = f(v)f(w)$, Lemma \ref{l:i(f)=I(fxf)} tells us that  $\iota(0) = 0$. Lemma \ref{l:I(F)>2i(pi(F))} tells us that $\iota(t) \leq 0$ for all $t \in (-\delta,\delta)$. Therefore $\iota(t)$ achieves its maximum at $t=0$ and we deduce that $\iota'(0) = 0$.

If is also possible, although it looks rather magical, to justify the equality by writing $\langle i'(f),q(f) \rangle$ as a double integral and symmetrizing the expression in the correct way.

Since \eqref{e:fromitoI} is in fact an equality, the inequality \eqref{e:question} turns out to be equivalent to the statement of Theorem \ref{t:main}. Our analysis so far is sharp.
\end{remark}

\section{Flowing along vector fields}
\label{s:vectorsb}

In this section we decompose the lifted Landau operator \eqref{e:Q} as a sum of second order differential operators each acting along one direction in $\R^6$. The vector fields $\tilde b_k$ giving these directions correspond to the generators of rotations along three perpendicular axes in $\R^3$ (which are then ``lifted'' to $\R^6$).  We write the majority of the computations in this paper using the vectors $b_k$. Because of this, it is important to get used to their basic arithmetic properties which we describe in this section.

First, for each value of $v-w$, we define the following three vectors
\begin{equation} \label{e:bk}
\begin{aligned}
b_1(v-w) = \begin{pmatrix} 0 \\ w_3 - v_3 \\ v_2 - w_2 \end{pmatrix}, \\
b_2(v-w) = \begin{pmatrix} v_3-w_3 \\ 0 \\ w_1-v_1 \end{pmatrix}, \\
b_3(v-w) = \begin{pmatrix} w_2-v_2 \\ v_1-w_1 \\ 0 \end{pmatrix}.
\end{aligned}
\end{equation}
One can check that the vector fields $b_1$, $b_2$ and $b_3$ are always perpendicular to $(v-w)$. In fact, for each fixed value of $v-w$ with $v\neq w$ these three vectors span the plane in $\mathbb{R}^3$ perpendicular to $v-w$. The following identity, that is readily verified by a direct computation, links these vectors with the Landau equation.
\begin{equation} \label{e:decompositionofaij} 
b_1 \otimes b_1 + b_2 \otimes b_2 + b_3 \otimes b_3 = a_{ij}(v-w),
\end{equation}
where $a_{ij}(z) = |z|^2 \delta_{ij} - z_i z_j$ is the matrix used in the definition of the Landau operator \eqref{e:landauoperator}. As a consequence of these two observations, the lifted operator $Q$ \eqref{e:Q} can be rewritten in the following way
\begin{align*}
Q(F) & = \sum_{k=1,2,3} (\dv_v- \dv_w)(\alpha(|v-w|)b_k\otimes b_k (\nabla_v - \nabla_w) F)\\
  & = \sum_{k=1,2,3} (\partial_{v_i}-\partial_{w_i}) \left( \alpha(|v-w|) (b_k)_i \ (b_k)_j \cdot (\partial_{v_j}-\partial_{w_j}) F \right).
\end{align*}

\begin{remark}\label{r:orthogonal projection identity}
Denote by $\Pi(z)$ ($z\neq 0$) the orthogonal projection from $\R^3$ to the plane perpendicular to $z$. It is elementary that $\Pi_{ij}(z) = \delta_{ij}- |z|^{-2}z_iz_j$ and so we see that $a_{ij}(z) = |z|^2\Pi(z)$. Accordingly the formula relating $a_{ij}$ and the vector fields $b_k$'s gives a corresponding one for $\Pi(z)$ which we record here
\begin{align*}	
\Pi_{ij}(z) = |z|^{-2}\sum \limits_{k=1,2,3}(b_k(z))_i(b_k(z))_j.
\end{align*}	
\end{remark}

Since the differential operators $b_k \cdot (\nabla_v - \nabla_w)$ act on functions in $\R^6$ it is both convenient and natural to write it in terms of vectors in $\R^6$. For $k=1,2,3$, we define corresponding vector fields $\tilde b_k$ in $\R^6$.
\begin{equation} \label{e:tildebk}
\begin{aligned}
\tilde b_k(v-w) = \begin{pmatrix} b_k \\ -b_k \end{pmatrix}.
\end{aligned}
\end{equation}
We abuse notation by writing $\tilde b_k : \R^6 \to \R^6$ to denote $\tilde b_k = \tilde b_k(v-w)$. The vector fields $\tilde b_1$, $\tilde b_2$ and $\tilde b_3$ are divergence-free in $\R^6$. They satisfy the following matrix identity,
\begin{equation} \label{e:lifteddecompositionofaij} 
\tilde b_1 \otimes \tilde b_1 + \tilde b_2 \otimes \tilde b_2 + \tilde b_3 \otimes \tilde b_3 = \begin{pmatrix} \{a_{ij}(v-w)\} & -\{a_{ij}(v-w)\} \\ -\{a_{ij}(v-w)\} & \{a_{ij}(v-w)\}\end{pmatrix}.
\end{equation}
Moreover, $\tilde b_k(v-w)$ is always perpendicular to $\begin{pmatrix} v-w \\ w-v \end{pmatrix}$. In particular, for any function $\beta = \beta(|v-w|)$, we always have (for the gradient and divergence in $\R^6$)
\begin{align*}
  \tilde b_k \cdot \nabla \beta = 0 \text{ and } \dv(\beta(|v-w|) \tilde b_k)=0.
\end{align*}
In light of all the above we can write $Q$ in terms of these vector fields $\tilde b_k$.
\[ \begin{aligned} 
Q(F) &= \sum_{k=1,2,3} \dv \left( \alpha(|v-w|)\left (\tilde b_k \cdot \nabla F\right ) \ \tilde b_k  \right) \\
&= \sum_{k=1,2,3} \alpha(|v-w|)\left (\tilde b_k \cdot \nabla \left( \ \tilde b_k \cdot \nabla F \right) \right ) \\
&= \sum_{k=1,2,3} \sqrt \alpha \ \tilde b_k \cdot \nabla \left( \sqrt \alpha \ \tilde b_k \cdot \nabla F \right).
\end{aligned}
\]
In this last formula, $\nabla F$ stands for the gradient of $F$ with respect to the variable $(v,w) \in \R^6$. We think of the operator $Q(F)$ as the sum of the second derivatives of $F$ along the directions $\tilde b_k$, with a weight $\alpha(|v-w|)$. To make this point clearer while also introducing additional notation we will use later, let us define, for a given vector field $b$ in $\R^6$, the operator
\begin{align*}
  L_b(F):= b \cdot \nabla F.
\end{align*}
Then $L_b(F)$ is simply the derivative of $F$ along $b$. One sees that the above decomposition for $Q$ can be restated as a sum of squares of operators of the form $L_{b}$, concretely
\begin{align}\label{e:Q_as_sum_of_squares_of_Lbk}
  Q = \sum \limits_{k=1,2,3} L_{\sqrt \alpha \tilde b_k}\circ L_{\sqrt \alpha \tilde b_k} = \sum \limits_{k=1,2,3}\alpha L_{\tilde b_k}\circ L_{\tilde b_k}.
\end{align}
It is interesting to notice that if we consider the flow of each vector field $\tilde b_k$,
\[ \begin{pmatrix} v'(t) \\ w'(t) \end{pmatrix} = \tilde b_k(v(t)-w(t)),\]
then $v(t)+w(t)$, $|v(t)|^2+|w(t)|^2$ and $|v(t)-w(t)|$ are all constant in $t$. This is simply the fact that $\tilde b$ is everywhere perpendicular to the gradient of $|v|^2+|w|^2$, the gradient of $|v-w|$, and to vectors of the form $(e,e) \in \R^6$ with $e \in \R^3$.

Given any point $(v,w) \in \R^6$, the set of points in $\R^6$ that we may reach by flowing along the three vector fields $b_k$'s, is exactly the sphere used in the usual formulas for the Boltzmann equation.
\begin{equation} \label{e:sphere}
 \sphere(v,w) := \left\{ (v',w') \in \R^6 : v'+w' = v+w \text{ and } |v'|^2 + |w'|^2 = |v|^2 + |w|^2\right\}.
\end{equation}
Based on this intuition, there are some quantities that we define later on taking into account derivatives along the directions $\tilde b_k$ for which we use names referring to this sphere. The set $\sphere(v,w)$ is indeed a two-dimensional sphere of radius $|v-w|/2$ isometrically embedded in $\R^6$. It is the same sphere used in the usual parametrization of the Boltzmann equation: 
\[ \sphere(v,w) = \left\{ (v',w'): v' = \frac{v+w}2 + \frac{|v-w|}2 \sigma, \ w' = \frac{v+w}2 - \frac{|v-w|}2 \sigma, \ \sigma \in S^2 \right\}.\]
At any given point $(v,w) \in \R^6$, the vectors $\tilde b_1(v-w)$, $\tilde b_2(v-w)$ and $\tilde b_3(v-w)$ are tangent to $\sphere(v,w)$. These vectors are necessarily linearly dependent. We can easily verify this fact by hand since $(v_1-w_1)\tilde b_1 + (v_2-w_2)\tilde b_2 + (v_3-w_3)\tilde b_3 = 0$.

Using the vector fields $\tilde b_i's$, we will eventually verify that the operator $Q$ is exactly the Laplace-Beltrami operator on $\sphere(v,w)$ times the function $|v-w|^2 \alpha(|v-w|)$. See Remark \ref{r:sphere}.

There is one last vector field in $\R^6$ that will play an important role in our analysis. Let $\n: \R^6 \to \R^6$ be the unit normal vector to the level sets of $|v-w|$. That is
\begin{equation} \label{e:n}
\n = \frac 1 {\sqrt{2} |v-w|} \begin{pmatrix} v-w \\ w-v \end{pmatrix}.
\end{equation}
We can see that for $k=1,2,3$, $[\n,\tilde b_k] = 0$. For the reader's convenience let us recall the definition of the Lie bracket $[a,b]$, this is the vector whose components are given by
\[ [a,b]_i = a_j \partial_j b_i - b_j \partial_j a_i.\]
The Lie bracket computes the error of commuting the differentiation with respect to $a$ and $b$,
\begin{equation} \label{e:liebracket}
 a \cdot \nabla (b \cdot \nabla u) - b \cdot \nabla (a \cdot \nabla u) = [a,b] \cdot \nabla u
\end{equation}
In contrast to the vectors $\tilde b_k$, in general we have $\n \cdot \nabla \alpha \neq 0$ whenever $ \alpha = \alpha(|v-w|)$. Moreover, we have
\[ [\n,\alpha \tilde b_k] = \left( \n \cdot \nabla \alpha \right) \ \tilde b_k = \sqrt{2} \alpha'(|v-w|) \, \tilde b_k. \]
For the last identity, note that
\begin{align*}
\n \cdot \nabla \alpha &= \n \cdot \frac {\alpha'(|v-w|)} {|v-w|} \begin{pmatrix} v-w \\ w-v \end{pmatrix} \\
 &= \frac {1} {\sqrt 2|v-w|} \begin{pmatrix} v-w \\ w-v \end{pmatrix} \cdot \frac {\alpha'(|v-w|)} {|v-w|} \begin{pmatrix} v-w \\ w-v \end{pmatrix} \\
 &= \frac{2 |v-w|^2}{\sqrt{2} |v-w|^2 } \alpha'(|v-w|) \\
 &= \sqrt{2} \alpha'(|v-w|).
\end{align*}

We will also be using vector fields to decompose the Fisher information, and we record here a definition and two preparatory lemmas that will be used later. Given a vector field $e:\mathbb{R}^{6}\to\mathbb{R}^{6}$ we define the functional
\begin{align*}
  I_e(F):= \iint_{\R^6} \frac{(e\cdot \nabla F)^2}F\dd w \dd v.
\end{align*}
One can see $I_e(F)$ as a kind of Fisher information\footnote{in fact, $I_e(F)$ is the standard Fisher information (as understood in statistics) at $\theta=0$ for the one-parameter distribution $F_\theta$ obtained by transporting the distribution $F$ along the flow given by $e$.} of $F$ that only uses $\nabla F$ in the direction $e$.  The first of the preparatory lemmas provides a formula for the Gateaux derivative of $I_e$ at $F$ (for some vector field $e$) in the direction given by $L_b(F)$ (for some divergence-free vector field $b$).
\begin{lemma}\label{l:Derivative_Ie_along_Lb}
Let $e$ and $b$ be vector fields in $\R^6$ and assume $b$ satisfies $\dv(b)=0$, then the following identity holds for any smooth positive function $F:\R^6 \to (0,\infty)$ with rapid decay at infinity.
\begin{align*}
\langle I_e'(F) , L_b(F) \rangle & = 2\iint_{\R^6} (e\cdot \nabla \log F)([e,b]\cdot \nabla \log F)F\dd w\dd v. 
\end{align*}

\end{lemma}

\begin{proof}
We compute directly, 
\begin{align*}
\langle I_e'(F) , L_b(F) \rangle = \iint_{\R^6} 2\frac{(e\cdot \nabla F)(e\cdot \nabla (b\cdot \nabla F))}F - \frac{(e\cdot \nabla F)^2}{F^2} (b\cdot \nabla F) \dd w \dd v.
\end{align*}
Observe that
\begin{align*}   
  e\cdot \nabla (b\cdot \nabla F) = b\cdot \nabla (e\cdot \nabla F) + [e,b]\cdot \nabla F.
\end{align*}
This results in
\begin{align*}
\langle I_e'(F) , L_b(F) \rangle = & \iint_{\R^6} 2\frac{(e\cdot \nabla F)(b\cdot \nabla (e\cdot \nabla F))}F - \frac{(e\cdot \nabla F)^2}{F^2} (b\cdot \nabla F) + 2\frac{(e\cdot \nabla F)([e,b]\cdot \nabla F)} F\dd w \dd v.
\end{align*}
The first two terms in the integrand add up to a divergence since
\begin{align*} 
2\frac{(e\cdot \nabla F)(b\cdot \nabla (e\cdot \nabla F))}F - \frac{(e\cdot \nabla F)^2}{F^2} (b\cdot \nabla F) &= b \cdot \nabla \frac{(e \cdot \nabla F)^2}F  \\
&= \dv \left( \frac{(e \cdot \nabla F)^2}F \ b \right).
\end{align*}  
It follows that
\begin{align*}
\langle I_e'(F) , L_b(F) \rangle = 2\iint_{\R^6} \frac{(e\cdot \nabla F)([e,b]\cdot \nabla F)} F \dd w \dd v.
\end{align*}  
and the lemma now follows using that $\nabla F = F\nabla \log F$.  
\end{proof}

The second preparatory lemma is the corresponding pure second derivative for $I_e(F)$.
\begin{lemma}\label{l:SecondDerivative_Ie_along_Lb}
Let $e$ and $b$ be vector fields in $\R^6$, then the following identity holds for any smooth positive function $F:\R^6 \to (0,\infty)$ with rapid decay at infinity.
\begin{align*}
\langle I_e''(F) L_b(F) , L_b(F) \rangle & = 2\iint_{\R^6} (e\cdot \nabla (b \cdot \nabla \log F))^2 F\dd w\dd v. 
\end{align*}
\end{lemma}

\begin{proof}
For any function $G : \R^6 \to \R$, we see that
\[ \langle I_e'(F), G \rangle = \iint_{\R^6} 2 \frac{(e \cdot \nabla F) (e \cdot \nabla G)}F - \frac{(e \cdot \nabla F)^2}{F^2} G \dd w \dd v.\]
Differentiating again, we get
\begin{align*}
\langle I''(F)G, G \rangle &= \iint_{\R^6} 2 \frac{(e \cdot \nabla G)^2}F - 4 \frac{(e \cdot \nabla F)(e \cdot \nabla G)}{F^2}G + 2  \frac{(e \cdot \nabla F)^2}{F^3} G^2 \dd w \dd v \\
&= \iint_{\R^6} 2 F \left( \frac{e \cdot \nabla G}F - G \frac{(e \cdot \nabla F)}{F^2}\right)^2 \dd w \dd v.
\end{align*}
In particular, if $G = L_b(F) = b \cdot \nabla F$, it reduces to
\begin{align*}
\langle I''(F) L_b(F), L_b(F) \rangle &= \iint_{\R^6} 2 F \left( \frac{e\cdot \nabla (b \cdot \nabla F)}F - \frac{(b\cdot \nabla F)} F \frac{(e \cdot \nabla F)}F\right)^2 \dd w \dd v \\
&= \iint_{\R^6} 2 F \left| e \cdot \nabla \left(b \cdot \frac{\nabla F} F \right) \right|^2 \dd w \dd v.
\end{align*}

\end{proof}
}

\section{The case of Maxwell molecules}
\label{s:mm}

In the case $\alpha(|v-w|) \equiv 1$ the monotonicity of the Fisher information was obtained in \cite{villani2000fisherlandau}, and also follows from the earlier work for the Boltzmann equation in \cite{toscani1992,villani1998fisherboltzmann}. Here, we provide a quick proof by verifying \eqref{e:question} in that case. It also serves to explain some of the ideas that will be used later for more general interaction potentials $\alpha$. Let $\tilde b_1$, $\tilde b_2$ and $\tilde b_3$ be the vector fields as in \eqref{e:tildebk}. We use the formula \eqref{e:Q_as_sum_of_squares_of_Lbk} for $\alpha \equiv 1$.
\[ Q(F) = \sum_{k=1,2,3} L_{\tilde b_k} \circ L_{\tilde b_k} (F). \]
We see that when $\alpha \equiv 1$ the operator $Q(F)$ is a sum of squares of differential operators given by vector fields which are generators of certain rotations in $\R^6$. Since the Fisher information is invariant under rotations, it follows that $I(F)$ is unchanged if one flows $F$ by $L_{\tilde b_k}F$. 
\begin{lemma} \label{l:flowofI}
For any smooth positive function $F:\R^6 \to (0,\infty)$ with rapid decay at infinity we have
\[ \langle I'(F) , L_{\tilde b_k}(F) \rangle = 0. \]
\end{lemma}

\begin{proof}
The result already follows by the rotational invariance, but we also verify directly by differentiating the integral. We proceed by direct computation. 
\begin{align*}
\langle I'(F) , L_{\tilde b_k}(F) \rangle &= \iint_{\R^6} 2 \frac{\nabla F \cdot \nabla L_{\tilde b_k}(F)}F - \frac{|\nabla F|^2}{F^2} L_{\tilde b_k}(F) \dd w \dd v \\
&= \iint_{\R^6} \frac{\tilde b_k \cdot \nabla |\nabla F|^2 + 2 \langle (D\tilde b_k) \nabla F, \nabla F \rangle }F - \frac{|\nabla F|^2}{F^2} L_{\tilde b_k}(F) \dd w \dd v \\
\intertext{Observe that $D\tilde b_k$ is antisymmetric for $k=1,2,3$, thus $\langle (D \tilde b_k) \nabla F, \nabla F \rangle \equiv 0$. }
&= \iint_{\R^6} \frac{\tilde b_k \cdot \nabla |\nabla F|^2}F - \frac{|\nabla F|^2}{F^2} \tilde b_k \cdot \nabla F \dd w \dd v \\
&= \iint_{\R^6} \tilde b_k \cdot \nabla \frac{|\nabla F|^2}F \dd w \dd v = 0.
\end{align*}
\end{proof}

Motivated by Lemma \ref{l:flowofI}, we study the first order transport equation where we flow $F$ by $L_{\tilde b_k}$. For any function $F_0 = F_0(v,w)$, let us consider the initial value problem
\begin{equation} \label{e:transport}
\begin{aligned}
F^k(0,v,w) &= F_0(v,w), \\
F_t^k(t,v,w) &= L_{\tilde b_k}(F).
\end{aligned}
\end{equation}

\begin{lemma} \label{l:Iconstantonflow}
Let $F^k$ be as in \eqref{e:transport}, then $I(F)$ is constant in $t$.
\end{lemma}
\begin{proof}
It follows from Lemma \eqref{l:flowofI}.
\end{proof}

\begin{lemma} \label{l:Qassecondderivative}
For any smooth positive function $F_0:\R^6 \to (0,\infty)$ with rapid decay at infinity, if $F^k$ solve the problems \eqref{e:transport} for $k=1,2,3$, we have
\[ Q(F_0)(v,w) = \sum_{k=1}^3 \partial_{tt} F^k(0,v,w) .\]
\end{lemma}

\begin{proof}
This is simply the fact that $Q = \sum_k L_{\tilde b_k} \circ L_{\tilde b_k}$ as in \eqref{e:Q_as_sum_of_squares_of_Lbk}.
\end{proof}

The following result was proved by Villani in \cite{villani2000fisherlandau}. Using our current framework, we are able to provide a short proof.

\begin{prop} \label{p:mm}
The Fisher information is monotone decreasing along the flow of the Landau equation \eqref{e:landauequation} when $\alpha \equiv 1$.
\end{prop}

\begin{proof}
We have to verify that \eqref{e:question} holds in this case. Then the result follows using Lemma \ref{l:dti=dtI/2}.

Let $F^k(t,v,w)$ solve \eqref{e:transport} with initial data $F_0 = F(v,w)$. We differentiate $I(F^k)$ twice using Lemma \ref{l:Iconstantonflow} to get that 
\[ 0 = \partial_{tt} I(F^k) = \langle I'(F^k), \partial_{tt} F^k \rangle + \langle I''(F^k) \partial_t F^k, \partial_t F^k \rangle. \]

Since the Fisher information is convex, then $\langle I''(F^k) \partial_t F^k, \partial_t F^k \rangle \geq 0$. Thus, using Lemma \ref{l:Qassecondderivative}, we conclude
\[ \langle I'(F), Q(F) \rangle = - \sum_{k=1,2,3} \langle I''(F) \partial_t F^k, \partial_t F^k \rangle_{|t=0} \leq 0.\]
\end{proof}
Lemma \ref{l:Qassecondderivative} still holds if we use $\sqrt \alpha \ \tilde b_k$ instead of $\tilde b_k$ for problem \eqref{e:transport}. However, Lemma \ref{l:Iconstantonflow} would not hold for any non-constant function $\alpha$. In order to verify the inequality \eqref{e:question}, we need another idea whenever $\alpha$ is not constant.

%
%
%
%
%
%


\section{The Fisher information along layers}
\label{s:alternatives}

As we explained in Section \ref{s:vectorsb} and recorded in \eqref{e:Q_as_sum_of_squares_of_Lbk}, the operator $Q$ can be expressed in terms of a composition of transport operators $L_{\sqrt{\alpha}\tilde b_k}$. In Section \ref{s:mm} we saw how this decomposition can be used to show the monotonicity of the Fisher information in the case of Maxwell molecules $\alpha \equiv 1$. In this section we begin the analysis for non-constant $\alpha$. 

Since we will be writing $L_{\sqrt{\alpha}\tilde b_k}$ repeatedly, from now on we will use the simpler notation
\begin{equation}\label{e:Lks}
  L_{k}(F) := \sqrt \alpha  \ \tilde b_k \cdot \nabla F.
\end{equation}
As noted in Section \ref{s:mm}, when $\alpha \equiv 1$ the operator $L_{k}$ is the generator of a flow by isometries in $\R^6$, and therefore it preserves the Fisher information $I(F)$. This is not true any more when $\alpha$ is not constant. We cannot use the analysis in Section \ref{s:mm} to immediately deduce \eqref{e:question}.

The main observation for this section is that since $\alpha$ depends only on $|v-w|$ and the vectors $\tilde b_k$ are tangent to these level sets, the flow defining $L_{k}$ is still an isometry layer-by-layer when we restrict our analysis to the level sets of $|v-w|$. Therefore, if we modify the Fisher information to only take into account the components of $\nabla F$ that are tangent to these level sets, we obtain a quantity that is indeed preserved by the flow of $L_{k}$.

The unit normal to the level set of $|v-w|$ is precisely the vector field $\n$ defined in \eqref{e:n}. The \emph{modified} Fisher information that we want to study is the following
\begin{equation} \label{e:Itangent}
\begin{aligned}
I_{tan}(F) &:= \iint_{\R^6} \frac{|\nabla F|^2}F - \frac{(\n \cdot \nabla F)^2} F \dd w \dd v \\
&= \iint_{\R^6} \frac{|\partial_{v_i} F|^2 + |\partial_{w_i} F|^2}F - \frac {(v_i-w_i) (v_j-w_j)}{2|v-w|^2} \frac{(\partial_{v_i} - \partial_{w_i})F (\partial_{v_j} - \partial_{w_j})F} F \dd w \dd v.
\end{aligned}
\end{equation}
It is convenient to write $I_{tan}$ as $I_{sph}+I_{par}$ where
\begin{align*}
I_{sph}(F) &:= \iint_{\R^6} \frac{a_{ij}(v-w)}{2|v-w|^2} \cdot \frac{(\partial_{v_i} - \partial_{w_i})F (\partial_{v_j} - \partial_{w_j})F} F \dd w \dd v, \\
I_{par}(F) &:= \iint_{\R^6} \frac{|(\partial_{v_i} + \partial_{w_i}) F|^2} {2F} \dd w \dd v.
\end{align*}
We used the notation $a_{ij}(z) = |z|^2 \delta_{ij} - z_i z_j$. 

\begin{remark} The subindices in $I_{tan},I_{sph},$ and $I_{par}$ are meant to convey which directions are being included. We already noted $I_{tan}(F)$ is the Fisher information of $F$ in the tangential directions to the level sets of $(v,w)\mapsto \alpha(|v-w|)$, and so here we call it the \emph{tangential} (with respect to the level sets of $\alpha$) \emph{Fisher information}. Meanwhile, $I_{sph}(F)$ involves the directions tangent to $\sphere(v,w)$, we call it the \emph{spherical Fisher information}. Last but not least, $I_{par}(F)$ involves the directions in the $3$ dimensional linear subspace of $\mathbb{R}^6$ of ``parallel'' pairs of velocities $\{ (e,e) \in \mathbb{R}^6 \mid e \in \R^3\}$, we call it the \emph{parallel Fisher Information}.

\end{remark}

\begin{prop} \label{p:In}
Let $\tilde b_k$ be the vector fields from \eqref{e:tildebk}, let $\alpha = \alpha(|v-w|)$ be an arbitrary nonnegative function and let $L_{k}$ be as in \eqref{e:Lks}. For any smooth positive function $F:\R^6 \to (0,\infty)$ with rapid decay at infinity and any $k=1,2,3$, we have
\begin{align*}
\langle I_{par}'(F) , L_{k}(F) \rangle &= 0 \\
\langle I_{sph}'(F) , L_{k}(F) \rangle &= 0 \\
\langle I_{tan}'(F) , L_{k}(F) \rangle &= 0.
\end{align*}

Moreover,
\[ \langle I_{tan}'(F), Q(F) \rangle = - \sum_{k=1,2,3} \langle I_{tan}''(F) L_{k}(F), L_{k}(F) \rangle \leq 0,\]
and similar identities hold for $I_{sph}$ and $I_{par}$.
\end{prop}

\begin{proof}
Since $I_{tan} = I_{par} + I_{sph}$, then the result follows for $I_{tan}$ after we prove it for $I_{sph}$ and $I_{par}$.

Let us start with the case of $I_{par}$. Let us define the unit vectors $p_1$, $p_2$ and $p_3$ in $\R^6$ as
\[ p_1= \frac 1 {\sqrt 2} \begin{pmatrix} 1 \\ 0 \\ 0 \\ 1 \\ 0 \\ 0 \end{pmatrix}, \qquad p_2 =  \frac 1 {\sqrt 2} \begin{pmatrix} 0 \\ 1 \\ 0 \\ 0 \\ 1 \\ 0 \end{pmatrix}, \qquad p_3=  \frac 1 {\sqrt 2} \begin{pmatrix} 0 \\ 0 \\ 1 \\ 0 \\ 0 \\ 1 \end{pmatrix}.\]
We write $I_{par}$ in terms of these vectors. We have
\[ I_{par}(F) = \sum_{i=1,2,3}\iint_{\R^6}\frac{|p_i \cdot \nabla F|^2}F \dd w \dd v.\]
We observe that $[p_i,\sqrt{\alpha} \ \tilde b_k]=0$ for $i=1,2,3$ and $k=1,2,3$. Then, using Lemma \ref{l:Derivative_Ie_along_Lb}, we get that $\langle I_{par}'(F), L_k(F) \rangle = 0$.


We now move on to $I_{sph}$. Because of \eqref{e:lifteddecompositionofaij}, we can write $I_{sph}$ as
\[ I_{sph} = \sum_{i=1,2,3} \iint_{\R^6}  \frac 1 {2|v-w|^2}  \frac{|\tilde b_i \cdot \nabla F|^2}F \dd w \dd v.\]
Differentiation in the directions $\tilde b_i$ and $\tilde b_k$ does not commute when $i \neq k$. Their commutators are easy to compute
\[
[\tilde b_1,\tilde b_2] = -2\tilde b_3, \qquad [\tilde b_2,\tilde b_3] = -2\tilde b_1, \qquad [\tilde b_3,\tilde b_1] = -2\tilde b_2.
\]
Taking this into account, we compute $\langle I_{sph}'(F), L_k(F) \rangle$ using Lemma \ref{l:Derivative_Ie_along_Lb}
\begin{align*}
 \langle I_{sph}'(F), L_k(F) \rangle 
&= \sum_{i=1,2,3} \iint_{\R^6} \frac{\sqrt{\alpha}}{|v-w|^2} (\tilde b_i \cdot \nabla \log F) \cdot ([\tilde b_i, \tilde b_k] \cdot \nabla \log F) \ F \dd w \dd v = 0.
\end{align*}
The last identity holds because the two terms where $k \neq i$ cancel out. For example for $k=1$, we have $[\tilde b_2, \tilde b_1] = 2\tilde b_3$ and $[\tilde b_3, \tilde b_1] = -2\tilde b_2$, which makes the integrand
\[ \frac{2\sqrt{\alpha}}{|v-w|^2} F \ \left( (\tilde b_2 \cdot \nabla \log F) \cdot (\tilde b_3 \cdot \nabla \log F) - (\tilde b_3 \cdot \nabla \log F) \cdot (\tilde b_2 \cdot \nabla \log F) \right) = 0. \]
Therefore, we obtain both $\langle I_{sph}'(F), L_{k}(F) \rangle = 0$ and $\langle I_{par}'(F), L_{_k}(F) \rangle$. Adding them, we also obtain $\langle I_{tan}'(F), L_{k}(F) \rangle$.

The final identity follows mimicking the proof of Proposition \ref{p:mm}.
\end{proof}

\section{Using commutators to estimate the remaining directional derivative}
\label{s:liebrackets}

In Section \ref{s:alternatives} we found a quantity $I_{tan}$ that is monotone decreasing in the direction of $Q$. The difference between this quantity $I_{tan}$ and the full Fisher information $I$ depends only on the component of $\nabla F$ in the single direction $\n$ perpendicular to the level sets of $|v-w|$. In this section, we analyze the value of $J := I - I_{tan}$ and compute its derivative explicitly.

\begin{lemma} \label{l:liebrackets}
Let $\n$ be as in \eqref{e:n} and $b = \tilde b_k$ for one of the vector fields \eqref{e:tildebk} with $k=1,2,3$. Let $J$ be the functional
\[ J(F) := \iint_{\R^6} \frac{|\n\cdot \nabla F|^2}F \dd w \dd v,\]
and $Q_b(F) := \alpha \, b \cdot \nabla(b \cdot \nabla F)$, for some scalar function $\alpha = \alpha(|v-w|)$.
The following identity holds
\[
\begin{aligned}
\langle J'(F),Q_b(F) \rangle &= \int_{\R^6}  - 2 F \left( \n \cdot \nabla ( \sqrt{\alpha} \ b \cdot \nabla \log F) \right)^2 \\
& \qquad + \frac{(\alpha')^2}{\alpha} (b \cdot \nabla \log F)^2 \ F  \dd w \dd v.
\end{aligned}
\]

\end{lemma}


\begin{proof}
The proof is a direct computation. We write it in detail.

Recall that $\dv (\alpha b)=0$ and also $\dv \left( (\n \cdot \nabla \alpha) b \right) = 0$. Moreover,
\[ [n, \alpha b] = (\n \cdot \nabla \alpha) b.\]

It is good to remember that since $[\n,b]=0$, then $\n \cdot \nabla(b \cdot \nabla G) = b\cdot \nabla(\n \cdot \nabla G)$ for any function $G$.

We differentiate and follow the computation
\begin{align*}
\langle J'(F),Q_b(F) \rangle &= \iint_{\R^6} 2 (\n \cdot \nabla \log F) (\n \cdot \nabla Q_b(F)) - (\n \cdot \nabla \log F)^2 Q_b(F) \dd w \dd v \\
&= \iint_{\R^6}  2 (\n \cdot \nabla \log F) \ (\n \cdot \nabla (\alpha \ b \cdot \nabla (b \cdot \nabla F))) - (\n \cdot \nabla \log F)^2 \ (\alpha b \cdot \nabla (b \cdot \nabla F)) \dd w \dd v.
\end{align*}
We commute differentiation with respect to $\n$ and $\alpha b$ introducing an error term,
\begin{align*}
\langle J'(F),Q_b(F) \rangle  &= \iint_{\R^6}  2 (\n \cdot \nabla \log F) \ (\alpha b \cdot \nabla (\n \cdot \nabla (b \cdot \nabla F)))  \\
& \qquad +2 (\n \cdot \nabla \log F) \ (\n \cdot \nabla \alpha) \ b \cdot \nabla (b \cdot \nabla F))  \\
&\qquad - (\n \cdot \nabla \log F)^2 \ (\alpha b \cdot \nabla (b \cdot \nabla F)) \dd w \dd v, \\
\intertext{integrating by parts using that $\dv b=0$ and $\dv (\alpha b) = 0$,}
&= \iint_{\R^6}  - 2 \alpha b \cdot \nabla (\n \cdot \nabla \log F) \ (\n \cdot \nabla (b \cdot \nabla F))  \\
& \qquad - 2 (\n \cdot \nabla \alpha) \ b \cdot \nabla( \n \cdot \nabla \log F ) \ (b \cdot \nabla F)  \\
&\qquad + \alpha b \cdot \nabla (\n \cdot \nabla \log F)^2 \ (b \cdot \nabla F) \dd w \dd v.
\end{align*}
We expand the derivative in the last term and observe $\n \cdot \nabla (b \cdot \nabla \log F) \ F = (\n \cdot \nabla (b \cdot \nabla F)) - (\n \cdot \nabla \log F)(b \cdot \nabla F)$, from where we get
\begin{align*}
\langle J'(F),Q_b(F) \rangle  &= \iint_{\R^6}  - 2\alpha F \left( \n \cdot \nabla (b \cdot \nabla \log F) \right)^2 \\
& \qquad - 2 (\n \cdot \nabla \alpha) b \cdot \nabla \log F \ (\n \cdot \nabla (b \cdot \nabla \log F)) \ F\dd w \dd v.
\end{align*}
We complete squares and conclude this is equal to 
\begin{align*}
&\iint_{\R^6}  - 2 F \left( \sqrt{\alpha} \ \n \cdot \nabla ( b \cdot \nabla \log F) + \frac {(\n \cdot \nabla \alpha)} {2 \sqrt{\alpha}} \ b \cdot \nabla \log F \right)^2 + \frac 12 \frac{(\n \cdot \nabla \alpha)^2}{\alpha} (b \cdot \nabla \log F)^2 F \dd w \dd v \\
&= \iint_{\R^6}  - 2 F \left( \n \cdot \nabla ( \sqrt{\alpha} b \cdot \nabla \log F) \right)^2 + \frac{(\alpha')^2}{\alpha} (b \cdot \nabla \log F)^2 F  \dd w \dd v.
\end{align*}
\end{proof}

\begin{remark}
If is possible to compute $\langle J'(F), Q(F) \rangle$ with a computation that follows the ideas of Section \ref{s:mm} more closely. For that, we would have to compute $\langle J',L_k \rangle$ (which is no longer zero) in terms of commutators of $n$ and $b_k$. Then, following the ideas in Section \ref{s:mm}, we may proceed with computing the second derivative in time of $I(F)$ when $F_t = L_k(F)$ and derive the formula in Lemma \ref{l:liebrackets}.

It is also true that we can rewrite the proof of Proposition \ref{p:mm} and of Proposition \ref{p:In} using a more direct computation with commutators like the one presented here for Lemma \ref{l:liebrackets}.

We believe the method of Section \ref{s:mm} is more intuitive than the computation presented in this section. Yet, there might be some value in presenting alternative approaches to the computation.
\end{remark}

\section{The three distinct diffusion terms}
\label{s:threediffusionterms}

In Section \ref{s:alternatives}, we analyzed the derivative of $I_{tan}$ in the direction $Q(F)$ and obtained a negative value in terms of $I''_{tan}$. In Section \ref{s:liebrackets}, we analyzed the derivative of $J$ in the direction $Q(F)$ and obtained a negative value in terms of $J''$ and a positive error term. In this section, we analyze the derivative of the full Fisher information $I$ in the direction $Q(F)$, which results from adding the estimates we computed in previous sections.

We want to analyze the second derivative of the Fisher information.

\begin{lemma} \label{l:D2I}
Let $I$ be the Fisher information of a smooth positive function $F:\R^{6} \to (0,\infty)$, as in \eqref{e:fisherlifted}. Let $b : \R^{6} \to \R^{6}$ be any smooth vector field. Then
\[ \langle I''(F) (b \cdot \nabla F), (b \cdot \nabla F) \rangle = 2 \iint_{\R^6} F |\nabla (b \cdot \nabla \log F)|^2 \dd w \dd v.\]
\end{lemma}

\begin{proof}
We apply Lemma \ref{l:SecondDerivative_Ie_along_Lb} with the canonical basis of $\R^6$: $e_1, e_2, \dots, e_6$. We add the six corresponding identities and obtain the desired result.
\end{proof}

We apply Lemma \ref{l:D2I} with $b = \sqrt{\alpha} \ \tilde b_k$ for each $k=1,2,3$ and $\alpha = \alpha(|v-w|)$. These vector fields are divergence-free for any scalar function $\alpha$. We obtain.
\[ \langle I''(F) (\sqrt \alpha \ \tilde b_k \cdot \nabla F), (\sqrt \alpha \tilde b_k \cdot \nabla F) \rangle = 2 \iint_{\R^6} F |\nabla (\sqrt \alpha \ \tilde b_k \cdot \nabla \log F)|^2 \dd w \dd v.\]

\begin{lemma} \label{l:remainder}
Let $Q$ be the linear operator associated to the Landau equation as in \eqref{e:Q}. Let $I$ be the usual Fisher information. Then
\begin{align*}
 \frac 12 \langle I'(F), Q(F) \rangle &= -\sum_{k=1,2,3} \iint_{\R^6} \left|\nabla \left[ \sqrt \alpha \ \tilde b_k \cdot \nabla \log F \right] \right|^2 \ F \dd w \dd v \\
& \qquad + \sum_{k=1,2,3} \iint_{\R^6} \frac{\alpha'(|v-w|)^2}{2\alpha(|v-w|)} \left(\tilde b_k \cdot \nabla \log F \right)^2 \ F \dd w \dd v.
\end{align*}
Here, $\tilde b_k$ are the vector fields defined in \eqref{e:tildebk}.
\end{lemma}

Both terms are homogeneous of degree one in $F$. The second term vanishes when $\alpha$ is constant, which corresponds to the monotonicity of the Fisher information in the Maxwell-molecules case.

\begin{proof}
We start by recalling that $I$ is almost the same as $I_{tan}$ defined in Section \ref{s:alternatives} except for one extra term
\[ I(F) = I_{tan}(F) + \iint_{\R^6} F \ |\n \cdot \nabla \log F|^2 \dd w \dd v.\]
Here, $\n$ be the vector field defined in \eqref{e:n}.

Recall that with the vector fields $\tilde b_k$ defined in \eqref{e:tildebk},
\[ Q(F) = \sum_{k=1,2,3} \sqrt \alpha \ \tilde b_k \cdot \nabla \left(\sqrt \alpha \ \tilde b_k \cdot \nabla F\right).\]

We established the monotonicity of $I_{tan}$ along the flow of $Q$ in Proposition \ref{p:In}. Moreover, we computed, for $L_k(F) = \sqrt \alpha \ \tilde b_k \cdot \nabla F$,
\begin{align} \label{e:a1}
\langle I_{tan}'(F), Q(F) \rangle &= \sum_{k=1,2,3} -\langle I_{tan}''(F) L_k(F),L_k(F) \rangle.
\end{align}

Let us write $J := I - I_{tan}$. That is
\[ J(F) = \iint_{\R^6} F |\n \cdot \nabla \log F|^2 \dd w \dd v.\]

We now apply Lemma \ref{l:liebrackets} to with each vector field $\tilde b_k$ and add the identities. We get the following equality

\begin{equation} \label{e:a2}
\langle J'(F), Q(F) \rangle = \sum_{k=1,2,3} \langle J''(F) L_k(F),L_k(F) \rangle + \iint_{\R^6} \frac{\alpha'(|v-w|)^2}{\alpha(|v-w|)} F \left( \tilde b_k \cdot \nabla \log F \right)^2 \dd w \dd v.
\end{equation}

Adding \eqref{e:a1} and \eqref{e:a2}, we get
\begin{align*}
 \langle I'(F), Q(F) \rangle &= \sum_{k=1,2,3} \langle I''(F) L_k(F),L_k(F) \rangle + \iint_{\R^6} \frac{\alpha'(|v-w|)^2}{\alpha(|v-w|)} F \left( \tilde b_k \cdot \nabla \log F \right)^2 \dd w \dd v \\
 &= \sum_{k=1,2,3} -2\iint_{\R^6} F\left|\nabla \left[ \sqrt \alpha \ \tilde b_k \cdot \nabla \log F \right]\right|^2 \dd w \dd v \\
 &\qquad + \iint_{\R^6} \frac{\alpha'(|v-w|)^2}{\alpha(|v-w|)} F \left( \tilde b_k \cdot \nabla \log F \right)^2 \dd w \dd v.
\end{align*}
\end{proof}

\begin{lemma} \label{l:gradientdecomposition}
Given any differentiable function $G : \R^{6} \to \R$, the following identity holds at every point $(v,w) \in \R^{6}$.
\[ |\nabla G|^2 = \frac 12 \sum_{i=1,2,3} |(\partial_{v_i}+\partial_{w_i})G|^2 + |\n \cdot \nabla G|^2 + \frac 1{2|v-w|^2} \sum_{k=1,2,3} \left| \tilde b_k \cdot \nabla G\right|^2.\]
Here $\n$ is as in \eqref{e:n} and $\tilde b_k$ are as in \eqref{e:tildebk}.
\end{lemma} 

\begin{proof}
We start from the elementary identities
\begin{align*}
  ( (\partial_{v_i}\pm\partial_{w_i})G)^2 = (\partial_{v_i}G)^2\pm 2(\partial_{v_i}G)(\partial_{w_i}G)+(\partial_{w_i}G)^2.
\end{align*}
These can be added together, resulting in
\begin{align*}
  |\nabla G|^2 = \sum \limits_{i=1}^3(\partial_{v_i}G)^2+(\partial_{w_i}G)^2 = \frac{1}{2}\sum \limits_{i=1}^3( (\partial_{v_i}+\partial_{w_i})G)^2  + ( (\partial_{v_i}-\partial_{w_i})G)^2.
\end{align*}
The sum of the squares $((\partial_{v_i}-\partial_{w_i})G)^2$ is simply $|\nabla_v G-\nabla_w G|^2$, and we have
\begin{align*}
  |\nabla_v G-\nabla_w G|^2 = \left|\Pi(v-w)(\nabla_v G-\nabla_w G) \right|^2 + \left( \frac{v-w}{|v-w|} \cdot (\nabla_v G-\nabla_w G) \right)^2.
\end{align*}
where $\Pi(v-w)$ denotes the orthogonal projector from $\R^3$ onto the space perpendicular to $v-w$, which we already introduced in Remark \ref{r:orthogonal projection identity}. The formula for $\Pi$ in Remark \ref{r:orthogonal projection identity} says that
\begin{align*}
\left|\Pi(v-w)(\nabla_v G-\nabla_w G) \right|^2 &= \frac{1}{|v-w|^{2}}\sum \limits_{k=1,2,3} (b_k \cdot (\nabla_vG-\nabla_w G))^2, \\
&= \frac{1}{|v-w|^{2}}\sum \limits_{k=1,2,3}(\tilde b_k\cdot \nabla G)^2. 
\end{align*}

From the definition of $\n$ in \eqref{e:n}, it follows that
\begin{align*}
  \frac{1}{2|v-w|^2}((\nabla_v G-\nabla_w G)\cdot (v-w))^2 = |\n\cdot\nabla G|^2, 	
\end{align*}	
and the lemma is proved.
\end{proof}

Based on Lemma \ref{l:gradientdecomposition}, we may decompose the second derivative $I''$ in the direction of $\sqrt \alpha \ \tilde b_k \cdot \nabla F$ as the sum of three terms. 
\begin{align*}
  |\nabla (\sqrt \alpha \ \tilde b_k \cdot \nabla \log F)|^2 & = \frac{1}{2} \sum \limits_{i=1,2,3}|(\partial_{v_i}+\partial_{w_i})(\sqrt \alpha \ \tilde b_k \cdot \nabla \log F)|^2 \\
  & \;\;\;\; + |\n \cdot \nabla (\sqrt \alpha \ \tilde b_k \cdot \nabla \log F)|^2 + \frac{1}{2|v-w|^2}\sum \limits_{i=1,2,3}|\tilde b_i \cdot \nabla (\sqrt \alpha \ \tilde b_k \cdot \nabla \log F)|^2.
\end{align*}
We multiply this identity by $F$, integrate over $\R^6$ and add up the result for $k=1,2,3$. This, combined with Lemma \ref{l:remainder}, leads to the decomposition
\begin{align*}
  \frac 12 \langle I'(F), Q(F) \rangle = -D_{parallel} - D_{radial} - D_{spherical}  + \sum_{k=1,2,3}\iint_{\R^6} \frac{(\alpha')^2}{2\alpha} \ F \ |\tilde b_k \cdot \nabla \log F|^2\dd w \dd v,
\end{align*}
where
\begin{align}
D_{parallel} &:= \frac 12 \sum_{i,j=1,2,3} \iint_{\R^6} \alpha(|v-w|) F |(\partial_{v_i}+\partial_{w_i}) \tilde b_j \cdot \nabla \log F|^2 \dd w \dd v \nonumber \\
D_{radial} &:= \sum_{i = 1,2,3} \iint_{\R^6} F |\n \cdot \nabla \left( \sqrt \alpha \ \tilde b_i \cdot \nabla \log F \right)|^2 \dd w \dd v \nonumber \\
D_{spherical} &:= \sum_{i,j=1,2,3} \iint_{\R^6} \frac \alpha{2|v-w|^2} F | \tilde b_i \cdot \nabla ( \tilde b_j \cdot \nabla \log F) |^2 \dd w \dd v \label{e:Dspherical}
\end{align}

Let us also write
\begin{equation} \label{e:Rspherical} 
R_{spherical} := \sum_{k=1,2,3} \iint_{\R^6} \frac{\alpha}{|v-w|^2} \ F \ |\tilde b_k \cdot \nabla \log F|^2 \dd w \dd v.
\end{equation}

Note that since $\alpha$ depends only on $|v-w|$, then $\tilde b_k \cdot \nabla \alpha=0$ and $(\partial_{v_i}+\partial_{w_i})\alpha=0$. This is the reason why we can pull the factor $\alpha$ outside of the differentiation in the expressions for $D_{parallel}$ and $D_{spherical}$.

\begin{lemma} \label{l:newremainder}
The following inequality holds
\[ \frac 12 \langle I'(F), Q(F) \rangle \leq -D_{parallel} -D_{radial} - D_{spherical} + \sup_{r>0} \left( \frac{r^2 \alpha'(r)^2}{2 \alpha(r)^2}\right) R_{spherical}. \]
\end{lemma}

\begin{proof}
This is a direct consequence of Lemma \ref{l:remainder} in terms of the new notation and using that the integrand in $R_{spherical}$ is nonnegative.
\end{proof}

In order to prove Theorem \ref{t:main}, we want to control the positive term $R_{spherical}$ with the negative terms $-D_{parallel} -D_{radial} -D_{spherical}$. We are going to use only $D_{spherical}$, that involves second derivatives in the directions $\tilde b_i$.

Recall that starting from any point $(v,w) \in \R^6$, the flow of the vector fields $b_1$, $b_2$ and $b_3$ stays within the sphere \eqref{e:sphere}. The inequality $D_{spherical} \geq \Lambda R_{spherical}$ will be deduced as a consequence of an elementary and apparently new inequality for functions on the sphere that we present in the next section.

\section{An inequality for functions on the sphere}
\label{s:logpoincare}

The objective of this section is to prove the following lemma, which will be crucial for the proof of our main result. It is a Poincar\'e-like inequality involving the derivatives of the logarithm of a function on the sphere $S^2$. Properly interpreted (see Remark \ref{r:intrinsic} below), it corresponds to some form of the $\Gamma_2$-criterion of Bakry and Emery (see \cite[Proposition 5.7.3]{bakry-2014}) on the projective space. The lemma is about symmetric functions on $S^2$, which is the two-dimensional sphere in $\R^3$. Equivalently, it is a result about functions on the projective space $\mathbb RP^2$. The value of the constant (which is 19/4 below) is not optimal. A more precise value for this constant, as well as a generalization to higher dimensions, is obtained by Sehyun Ji in \cite{sehyun2024}.

\begin{lemma} \label{l:logpoincare}
Let $f : S^2 \to (0,\infty)$ be a $C^2$ function on the sphere. Assume that $f(\sigma) = f(-\sigma)$. The following inequality holds
\begin{align}\label{e:logpoincare} 
\sum_{i,j = 1,2,3} \int_{S^2} f (b_i \cdot \nabla (b_j \cdot \nabla \log f))^2 \dd \sigma \geq \frac{19}4 \sum_{i = 1,2,3} \int_{S^2} f (b_i \cdot \nabla \log f)^2 \dd \sigma.
\end{align}
\end{lemma}

\begin{remark}
Note that for $\sigma \in S^2$, the vector fields $b_1(\sigma)$, $b_2(\sigma)$, and $b_3(\sigma)$ are perpendicular to $\sigma$ and are therefore vectors on the tangent space of $S^2$ at every point $\sigma \in S^2$. It is useful to compute their Lie brackets.
\begin{equation} \label{e:commutatorsb}
[b_1,b_2] = -b_3, \qquad [b_2,b_3] = -b_1, \qquad [b_3,b_1] = -b_2.
\end{equation}

\end{remark}

\begin{remark}
For functions $\phi,\varphi \in C^1(S^2)$ their gradients $\nabla_\sigma \phi$ and $\nabla_\sigma \varphi$ at a $\sigma \in S^2$ are elements of the tangent plane to $S^2$. This makes them vectors in $\R^3$ perpendicular to $\sigma$. The directional derivative $b_i \cdot \nabla \varphi$ is an intrinsic differentiation of the function $\varphi$ on $S^2$. Both vectors $b_i(\sigma)$ and $\nabla_\sigma \varphi$ belong to the tangent space of $S^2$ at $\sigma$. Moreover, one can check that the $b_i's$ are also divergence free as vector fields on $S^2$. We omit the subindex $\sigma$ for $\nabla = \nabla_\sigma$ in most of this section.

Applying the decomposition \eqref{e:decompositionofaij}, we observe that $b_1 \otimes b_1 + b_2 \otimes b_2 + b_3 \otimes b_3$ is the orthogonal projector matrix to the tangent space to $S^2$ at $\sigma$. In particular, it is the identity when applied to vectors on that tangent space. We derive the following elementary but useful identity
\begin{align}\label{e:decompositionofaij_applied_to_S2}
  \sum\limits_{i=1,2,3}(b_i(\sigma )\cdot \nabla \phi)(b_i(\sigma)\cdot \nabla \varphi) = \nabla_\sigma \phi\cdot \nabla_\sigma \varphi \;\;\forall \phi,\varphi\in C^1(S^2).
\end{align}
Another useful identity that follows from this, and that we record here for later use relates the Laplace-Beltrami operator on $S^2$ with second derivatives in the directions $b_i$.
\begin{align}\label{e:laplacebeltrami}
  \sum \limits_{i=1,2,3}(b_i\cdot \nabla (b_i\cdot \nabla f)) = \Delta_\sigma f\;\;\forall\;f\in C^2(S^2).
\end{align}
To see this, we multiply the left hand side by $\varphi \in S^2$ and integrate by parts over $S^2$ using that the $b_i's$ are divergence free. We use \eqref{e:decompositionofaij_applied_to_S2} to obtain
\begin{align*}
  \sum \limits_{i=1,2,3}\int_{S^2}(b_i\cdot \nabla (b_i\cdot \nabla f))\varphi \dd \sigma = -\int_{S^2} \nabla_\sigma f\cdot \nabla_\sigma \varphi \dd \sigma.
\end{align*}
Integrating by parts on the second integral, it follows that
\begin{align*}
  \sum \limits_{i=1,2,3}\int_{S^2}(b_i\cdot \nabla (b_i\cdot \nabla f))\varphi \dd \sigma = \int_{S^2} \Delta_\sigma f \varphi\dd \sigma.
\end{align*}
Since $\varphi$ is an arbitrary function in $C^2(S^2)$, we have proved \eqref{e:laplacebeltrami}.

\end{remark}
The starting point for the proof of Lemma \ref{l:logpoincare} is to express both sides of \ref{e:logpoincare} in terms of $\sqrt f$, motivated by the well known observation that the Fisher information of $f$ is four times the $\dot H^1$ norm of $\sqrt f$. Once this is done, the lemma will follow from the Poincar\'e inequality on $S^2$ applied to each $(b_j \cdot \nabla_\sigma \sqrt f)$. 

It is easy to write the right hand side of \eqref{e:logpoincare} in terms of $\sqrt{f}$.  We have (using \eqref{e:decompositionofaij_applied_to_S2})
\begin{equation} \label{e:rhssquareroot}
\sum_{i = 1,2,3} \int_{S^2} f (b_i \cdot \nabla \log f)^2 \dd \sigma = 4 \sum_{i = 1,2,3} \int_{S^2} (b_i \cdot \nabla \sqrt f)^2 \dd \sigma = 4\int_{S^2}|\nabla_\sigma \sqrt{f}|^2\dd \sigma.
\end{equation}
As for the left hand side of \eqref{e:logpoincare}, for the terms with $i=j$ there is a simple trick to obtain a clean bound in terms of a respective integral involving $\sqrt f$. Indeed, consider the case $i=j$. Writing $b = b_i$, we see that
\begin{align}
\int_{S^2} \left( b \cdot \nabla (b \cdot \nabla \log f) \right)^2 f \dd \sigma &= 4 \int_{S^2} \left( b \cdot \nabla (b \cdot \nabla \sqrt f) \right)^2 \nonumber \\
&\qquad - 2 \frac{\left( b \cdot \nabla (b \cdot \nabla \sqrt f) \right) (b \cdot \nabla \sqrt f)^2}{\sqrt f} \nonumber \\
&\qquad + \frac{(b \cdot \nabla \sqrt f)^4}f \dd \sigma. \nonumber
\intertext{Integrating by parts the middle term, we arrive at}
\int_{S^2} \left( b \cdot \nabla (b \cdot \nabla \log f) \right)^2 f \dd \sigma &=4 \int_{S^2} \left( b \cdot \nabla (b \cdot \nabla \sqrt f) \right)^2 + \frac 13 \frac{(b \cdot \nabla \sqrt f)^4}f \dd \sigma \nonumber \\
& \geq 4 \int_{S^2} \left( b \cdot \nabla (b \cdot \nabla \sqrt f) \right)^2 \dd \sigma \label{e:trick_pure_second_derivatives}.
\end{align}

We can repeat this trick for all the terms when $i=j$, but it will fall short of producing a useful expression for the terms with $i \neq j$. To overcome this, we will consider pure second derivatives in all possible directions below. Let us first introduce, for a given $g \in C^2(S^2)$, the matrices 
\begin{equation} \label{e:MN}
(M_g)_{ij} = \frac {b_i \cdot \nabla (b_j \cdot \nabla g) + b_j \cdot \nabla (b_i \cdot \nabla g)} 2, \qquad  (N_g)_{ij} = \frac {b_i \cdot \nabla (b_j \cdot \nabla g) - b_j \cdot \nabla (b_i \cdot \nabla g)} 2.
\end{equation}
That is, $M_g(\sigma)$ and $N_g(\sigma)$ are the symmetric and antisymmetric parts of the matrix $(b_i \cdot \nabla(b_j \cdot \nabla g))_{ij}$. Note that $M_g(\sigma)$ and $N_g(\sigma)$ are $3\times 3$ matrices, corresponding to the fact that we are working with three vector fields $b_1,b_2,b_3$. Moreover, the squared norms of $M_g$ and $N_\sigma$ are of direct interest to us since
\begin{align*}
  \sum \limits_{i,j=1,2,3}(b_i \cdot \nabla_\sigma (b_j \cdot \nabla_\sigma \log f))_{ij}^2	= \sum \limits_{i,j=1,2,3} (M_{\log f})_{ij}^2+(N_{\log f})_{ij}^2 = \|M_{\log f}(\sigma)\|^2+\|N_{\log f}(\sigma)\|^2.
\end{align*}	
Therefore, we want to understand the integrals of $\|M_{\log f}(\sigma)\|^2f$ and $\|N_{\log f}(\sigma)\|^2f$. Analyzing the latter is straightforward since the components of $N_g(\sigma)$ are given by the derivatives $(b_i\cdot\nabla g)$.
\begin{lemma} \label{l:antisymmetric-commutators}
Let $g : S^2 \to \R$ be any $C^2$ function and $N_g(\sigma)$ as in \eqref{e:MN}, then
\[ \|N_g(\sigma)\|^2 = \frac 12 \sum_{i=1,2,3} (b_i \cdot \nabla g)^2.\]
\end{lemma}

\begin{proof}
The entries $(N_g)_{ij}$ correspond to the Lie brackets of $b_1,b_2,b_3$, $(N_g)_{ij} = \tfrac{1}{2}[b_i,b_j]$. Thus, 
\begin{align*}
\|N\|^2 &= \sum_{i,j = 1,2,3} N_{ij}^2 = \frac 14 \sum_{i,j = 1,2,3} ([b_i,b_j] \cdot \nabla g)^2, \\
\intertext{Note that each nonzero term appears twice in the sum. Using \eqref{e:commutatorsb} we get,}
&= \frac 12 \sum_{i = 1,2,3} (b_i \cdot \nabla g)^2.
\end{align*}
\end{proof}

Lemma \ref{l:antisymmetric-commutators} already gives us a fraction of the inequality of Lemma \ref{l:logpoincare}. Indeed, by Lemma \ref{l:antisymmetric-commutators}, we get
\begin{align}
\sum_{i,j = 1,2,3} \int_{S^2}  (b_i \cdot \nabla (b_j \cdot \nabla \log f))^2f \dd \sigma &= \int_{S^2} \|M_{\log f}(\sigma)\|^2 f(\sigma) \dd \sigma + \frac 12 \int_{S^2} \sum_{i=1,2,3} (b_i \cdot \nabla \log f)^2 f(\sigma) \dd \sigma. \nonumber
\end{align}


%

To obtain our desired inequality \eqref{e:logpoincare} it remains to understand the integral of $\|M_{\log f}\|^2f$. Recall that a symmetric matrix $A \in \R^{d \times d}$ is uniquely determined by the quadratic form $\langle Ae, e\rangle$. Thus, we may try to estimate its norm $\|A\|^2$ in terms of the values of $\langle Ae, e\rangle^2$ as $e$ ranges over all vectors in $S^2$. Likewise, in order to compute $\|M_{\log f}(\sigma)\|^2$, we may average the values of second derivatives of $\log f$ along all possible directions and properly normalize it. We need a practical way to determine all possible directional derivatives from a point $\sigma \in S^2$. For any $e \in S^2$, we define the vector field $b_e$ using the cross product.
\[ b_e(\sigma) = \sigma \times e.\]
Note that $b_{e_1}$, $b_{e_2}$ and $b_{e_3}$ conveniently correspond to $b_1$, $b_2$ and $b_3$ from \eqref{e:bk}. Using these vector fields we want to recover the quantities on the left and right hand sides of Lemma \ref{l:logpoincare}. It is easy for the right-hand side using the next lemma.

\begin{lemma} \label{l:alldirections1st}
For any $C^1$ function $g :S^2 \to \R$, the following equality holds at every point $\sigma \in S^2$.
\[ \sum_{i = 1,2,3} (b_i \cdot \nabla g(\sigma))^2 = \frac3{\omega_2} \int_{e \in S^2} (b_e \cdot \nabla g(\sigma))^2 \dd e. \]
\end{lemma}

\begin{proof}
Note that here $\sigma \in S^2$ is a fixed point and $\nabla g(\sigma)$ is the same vector for all values of $e \in S^2$. Let us write $e = (x_1,x_2,x_3)$. We observe that $b_e = x_1 b_1 + x_2 b_2 + x_3 b_3$ so
\begin{align*}
\int_{e \in S^2} (b_e \cdot \nabla g(\sigma))^2 \dd e &= \int_{e \in S^2} (\sum_{i=1,2,3} x_i b_i \cdot \nabla g)^2 \dd e \\
 &= \int_{e \in S^2} \sum_{i,j=1,2,3} (x_i b_i \cdot \nabla g) (x_j b_j \cdot \nabla g) \dd e.
 \intertext{Note that $b_i = b_i(\sigma)$ and $\nabla g = \nabla g(\sigma)$ are independent of the variable of integration $e \in S^2$. Since $x_i x_j$ integrates to zero on $S^2$ unless $i = j$, we get}
 \int_{e \in S^2} (b_e \cdot \nabla g(\sigma))^2 \dd e  &= \int_{e \in S^2} \sum_{i=1,2,3} x_i^2 (b_i \cdot \nabla g)^2 \dd e \\
 &= \left( \int_{e \in S^2} x_1^2  \dd e\right) \sum_{i=1,2,3} (b_i \cdot \nabla g)^2 \\
 &= \frac{\omega_2}3 \sum_{i=1,2,3} (b_i \cdot \nabla g)^2. 
\end{align*}
\end{proof}

The following is the second order version of Lemma \ref{l:alldirections1st}. It is slightly more complicated.

\begin{lemma} \label{l:alldirections2nd}
For any $C^2$ function $g :S^2 \to \R$ and any $\sigma \in S^2$ we have the identity,
\[
2c_1 \|M_g(\sigma)\|^2 + c_1 |\trace (M_g(\sigma))|^2 = \int_{e \in S^2} (b_e(\sigma) \cdot \nabla (b_e(\sigma) \cdot \nabla g(\sigma)))^2 \dd e.
\]
Here, 
\[ c_1 = \int_{S^2} x_1^2 x_2^2 \dd S. \]
\end{lemma}

From the identity \ref{e:laplacebeltrami}, we point out that $\trace (M_g) = \Delta_\sigma g$.

\begin{proof} 
We write $e = (x_1,x_2,x_3)$ and expand like in the proof of Lemma \ref{l:alldirections1st}.
\begin{align*}
\int_{e \in S^2} (b_e \cdot \nabla (b_e \cdot \nabla g))^2 \dd e &= \int_{e \in S^2} \sum_{i,j,k,l=1,2,3} (x_i b_i \cdot \nabla (x_j b_j \cdot \nabla g)) (x_k b_k \cdot \nabla (x_l b_l \cdot \nabla g)) \dd e \\
 &= \sum_{i,j,k,l=1,2,3} \left( \int_{e \in S^2} x_i x_j x_k x_l \dd e\right)  (b_i \cdot \nabla (b_j \cdot \nabla g)) (b_k \cdot \nabla (b_l \cdot \nabla g)).
\end{align*}
The integral factors vanish save when $(i,j,k,l)$ is either four equal indices, or two pairs of equal indices. We are left with
\begin{align*}
\int_{e \in S^2} (b_e \cdot \nabla (b_e \cdot \nabla g))^2 \dd e &= \left( \int_{e \in S^2} x_1^4 \dd e\right) \sum_{i=1,2,3}  (b_i \cdot \nabla (b_i \cdot \nabla g))^2 \\
&\qquad + \left( \int_{e \in S^2} x_1^2 x_2^2 \dd e\right) \sum_{i\neq j}  (b_i \cdot \nabla (b_j \cdot \nabla g))^2 \\
&\qquad + \left( \int_{e \in S^2} x_1^2 x_2^2 \dd e\right) \sum_{i\neq j}  (b_i \cdot \nabla (b_j \cdot \nabla g)) (b_j \cdot \nabla (b_i \cdot \nabla g)) \\
&\qquad + \left( \int_{e \in S^2} x_1^2 x_2^2 \dd e\right) \sum_{i\neq j}  (b_i \cdot \nabla (b_i \cdot \nabla g)) (b_j \cdot \nabla (b_j \cdot \nabla g)).
\end{align*}
We define $c_1 > 0$ to be the (computable\footnote{It can be seen by an elementary computation that $c_1 = 4 \pi / 15$}) value of the integral in the last factor
\[ c_1 := \frac 13 \int_{e \in S^2} x_1^4 \dd e = \int_{e \in S^2} x_1^2 x_2^2 \dd e.\]
We continue with our computation,
\begin{align*}
\int_{e \in S^2} (b_e \cdot \nabla (b_e \cdot \nabla g))^2 \dd e &= 3c_1 \sum_{i=1,2,3}  (b_i \cdot \nabla (b_i \cdot \nabla g))^2 \\
&\qquad + c_1 \sum_{i\neq j}  (b_i \cdot \nabla (b_j \cdot \nabla g))^2 \\
&\qquad + c_1 \sum_{i\neq j}  (b_i \cdot \nabla (b_j \cdot \nabla g)) (b_j \cdot \nabla (b_i \cdot \nabla g)) \\
&\qquad + c_1 \sum_{i\neq j}  (b_i \cdot \nabla (b_i \cdot \nabla g)) (b_j \cdot \nabla (b_j \cdot \nabla g)) \\
&= c_1 \sum_{i, j = 1,2,3}  (b_i \cdot \nabla (b_j \cdot \nabla g))^2 \\
&\qquad + c_1 \sum_{i, j = 1,2,3}  (b_i \cdot \nabla (b_j \cdot \nabla g)) (b_j \cdot \nabla (b_i \cdot \nabla g)) \\
&\qquad + c_1 (\sum_{i=1,2,3} b_i \cdot \nabla (b_i \cdot \nabla g))^2 \\
&= 2c_1 \| M_{g}(\sigma) \|^2 + c_1 |\trace (M_g(\sigma))|^2.
\end{align*}
\end{proof}

\begin{lemma} \label{l:singledirectioninequality}
For any $C^2$ function $f:S^2 \to (0,\infty)$ and any $e \in S^2$ the following inequality holds.
\[ \int_{S^2} \left( b_e \cdot \nabla (b_e \cdot \nabla \log f) \right)^2 f \dd \sigma \geq 4 \int_{S^2} \left( b_e \cdot \nabla (b_e \cdot \nabla \sqrt f) \right)^2 \dd \sigma.\]
\end{lemma}

\begin{proof}
We have already proved this in the discussion leading to \eqref{e:trick_pure_second_derivatives}. In this lemma, we state the inequality for the vector fields $b_e$.
\end{proof}

\begin{lemma} \label{l:symroot}
For any $C^2$ function $f:S^2 \to (0,\infty)$ the following inequality holds.
\[ \int_{S^2} \left( 2\|M_{\log f}(\sigma)\|^2 + |\trace (M_{\log f}(\sigma))|^2 \right) \ f(\sigma) \dd \sigma \geq 4 \int_{S^2} 2\|M_{\sqrt{f}}(\sigma)\|^2 + |\trace (M_{\sqrt{f}}(\sigma))|^2 \dd \sigma.\]
\end{lemma}

\begin{proof}
We use Lemma \ref{l:alldirections2nd} to rewrite the left-hand side as follows,
\begin{align*}
\int_{S^2} \left( 2\|M_{\log f}(\sigma)\|^2 + |\trace (M_{\log f}(\sigma))|^2 \right) \ f(\sigma) \dd \sigma &= \frac 1{c_1} \int_{S^2} \int_{S^2}  \left( b_e \cdot \nabla(b_e \cdot \nabla \log f(\sigma)) \right)^2 \ f(\sigma) \dd e \dd \sigma.
\end{align*}
We apply Lemma \ref{l:singledirectioninequality} for each individual direction $e \in S^2$ and see the integral on the right is no smaller than
\begin{align*}
\frac 1{c_1} \int_{S^2} \int_{S^2}  \left( b_e \cdot \nabla(b_e \cdot \nabla \log f(\sigma)) \right)^2 \ f(\sigma) \dd e \dd \sigma \geq \frac 4{c_1} \int_{S^2} \int_{S^2}  \left( b_e \cdot \nabla(b_e \cdot \nabla \sqrt f) \right)^2 \dd \sigma \dd e.
\end{align*}
The desired inequality then follows by applying Lemma \ref{l:alldirections2nd} again to rewrite this last integral
\begin{align*}
\frac 4{c_1} \int_{S^2} \int_{S^2}  \left( b_e \cdot \nabla(b_e \cdot \nabla \sqrt f) \right)^2 \dd \sigma \dd e =  4 \int_{S^2}  2\|M_{\sqrt{f}}(\sigma)\|^2 + |\trace(M_{\sqrt{f}}(\sigma))|^2 \dd \sigma.
\end{align*}
Combining the last three displayed formulas, we conclude the inequality of the Lemma.
\end{proof}

Thanks to Lemma \ref{l:symroot}, one can express the integral of $\|M_{\log f}\|^2f$ in terms of that of $\|M_{\sqrt{f}}\|^2$ plus some extra terms involving their respective traces. The next two lemmas will use this to obtain a further inequality, one without the inconvenient trace terms.
\begin{lemma} \label{l:linearalgebra}
Let $A \in \R^{3 \times 3}$ and let $M$ and $N$ be its symmetric and anti-symmetric parts, that is
\[ M = \frac{A+A^T}2 \qquad N = \frac{A-A^T}2. \]
Assume that $\rank A \leq 2$, then
\[ \|M\|^2 \geq \frac 12 |\trace (M)|^2.\]
\end{lemma}

\begin{proof}
The quantities $\|M\|$ and $\trace (A)$ (which is the same as $\trace M$) are invariant by orthonormal changes of coordinates. Since $\rank A \leq 2$, it has a zero eigenvector. We can pick an orthonormal basis that starts with this eigenvector, so that $A_{1,1} = 0$. In this basis,
\[ \|M\|^2 \geq A_{2,2}^2 + A_{3,3}^2 \geq \frac 12 (A_{2,2} + A_{3,3})^2 = \frac 12 |\trace (A)|^2 = \frac 12 |\trace (M)|^2.\]
\end{proof}

As we argue below, one can apply Lemma \ref{l:linearalgebra} to see that for every $\sigma \in S^2$, $2\|M(\sigma)\|^2 \geq |\trace M(\sigma)|^2$. We use this fact for the next lemma.

\begin{lemma} \label{l:lp2}
Let $f : S^2 \to (0,\infty)$ be a $C^2$ function. The following inequality holds,
\begin{align*}
\int_{S^2} \|M_{\log f}(\sigma)\|^2 \ f \dd \sigma \geq \int_{S^2} 2\|M_{\sqrt{f}}(\sigma)\|^2 + |\trace (M_{\sqrt{f}}(\sigma))|^2 \dd \sigma.
\end{align*}
\end{lemma}

\begin{proof}
The matrix $A(\sigma)$ given by $A_{ij}(\sigma) = (b_i \cdot \nabla(b_j \cdot \nabla \log f)$ is of rank at most two because the vectors $b_1$, $b_2$ and $b_3$ are linearly dependent at every $\sigma$.

From Lemma \ref{l:linearalgebra}, we deduce that $4\|M_{\log f}\|^2 \geq 2\|M_{\log f}\|^2 + |\trace (M_{\log f})|^2$ and then replace it on the left-hand side of Lemma \ref{l:symroot}.
\end{proof}

The inequality in Lemma \ref{l:lp2} is how we have managed to turn the trick in Lemma \ref{l:singledirectioninequality} into an inequality relating $M_{\log f}$ to $M_{\sqrt{f}}$ (plus an extra term). With \eqref{l:lp2} at hand, what remains in order to prove Lemma \ref{l:logpoincare} is using the integral involving $M_{\sqrt{f}}$ to bound the integrals of $(b_i\cdot \nabla \sqrt{f})^2$. 

The following lemma is the version for $M_g(\sigma)$ of the well-known fact that the integral of $|D^2 g|^2$ and $|\Delta g|^2$ coincide for any twice differentiable function $g: \mathbb T^d \to \R$.

\begin{lemma}\label{l:laplaciansquared}
Let $g : S^2 \to \R$ be any $C^2$ function. Then, the following identity holds.
\begin{align*}
\int_{S^2} \left( \sum_{i=1,2,3} b_i \cdot \nabla (b_i \cdot \nabla g) \right)^2 \dd \sigma &= \int_{S^2} \|M_g(\sigma)\|^2 \dd \sigma + \frac 12 \int_{S^2} \sum_{i=1,2,3} (b_i \cdot \nabla g)^2 \dd \sigma \\
&= \sum_{i,j = 1,2,3} \int_{S^2} (b_i \cdot \nabla (b_j \cdot \nabla g))^2 \dd \sigma.
\end{align*}
\end{lemma}

\begin{proof}
We integrate by parts starting from the left-hand side,
\begin{align*}
\int_{S^2} &\left( \sum_{i=1,2,3} b_i \cdot \nabla (b_i \cdot \nabla g) \right)^2 \dd \sigma = \sum_{i,j = 1,2,3} \int_{S^2} (b_i \cdot \nabla(b_i \cdot \nabla g)) (b_j \cdot \nabla(b_j \cdot \nabla g)) \dd \sigma \\
&= \sum_{i,j = 1,2,3} \int_{S^2} - (b_i \cdot \nabla g)  (b_i \cdot \nabla(b_j \cdot \nabla(b_j \cdot \nabla g)) \dd \sigma. \\
\intertext{We introduce commutators to switch the order of differentiation,}
&= \sum_{i,j = 1,2,3} \int_{S^2} - (b_i \cdot \nabla g)  (b_j \cdot \nabla(b_i \cdot \nabla(b_j \cdot \nabla g)) - (b_i \cdot \nabla g) \ ([b_i,b_j] \cdot \nabla(b_j \cdot \nabla g)) \dd \sigma. \\
\intertext{We integrate by parts again,}
&= \sum_{i,j = 1,2,3} \int_{S^2} (b_j \cdot \nabla(b_i \cdot \nabla g))  (b_i \cdot \nabla(b_j \cdot \nabla g)) - (b_i \cdot \nabla g) \ ([b_i,b_j] \cdot \nabla(b_j \cdot \nabla g)) \dd \sigma.\\
\intertext{We use that the antisymmetric part of the second derivatives correspond to differentiation along commutators $[b_i,b_j]$,} 
&= \sum_{i,j = 1,2,3} \int_{S^2} |(M_g)_{ij}|^2 - \frac {([b_i,b_j] \cdot \nabla g)^2}4  - (b_i \cdot \nabla g) \ ([b_i,b_j] \cdot \nabla(b_j \cdot \nabla g)) \dd \sigma. \\
\intertext{We apply another commutator to the second term and integrate by parts again,}
&= \sum_{i,j = 1,2,3} \int_{S^2} |(M_g)_{ij}|^2 - \frac {([b_i,b_j] \cdot \nabla g)^2}4\\
&\qquad\qquad\qquad - (b_i \cdot \nabla g) \ ([[b_i,b_j],b_j] \cdot \nabla g)) + b_j \cdot \nabla (b_i \cdot \nabla g) ([b_i,b_j] \cdot \nabla g) \dd \sigma.\\
\intertext{Note that $[[b_i,b_j],b_j] = -b_i$ whenever $i \neq j$. Moreover, exchanging $i$ with $j$ in the last term, the symmetric part of $b_j \cdot \nabla (b_i \cdot \nabla g)$ cancels out and we are left with,}
&= \sum_{i,j = 1,2,3} \int_{S^2} |(M_g)_{ij}|^2 - \frac {([b_i,b_j] \cdot \nabla g)^2}4 +\delta_{i \neq j} (b_i \cdot \nabla g)^2 - \frac {([b_i,b_j] \cdot \nabla g)^2} 2 \dd \sigma \\
&= \int_{S^2} \|M_g\|^2 \dd \sigma + \frac 12 \int_{S^2} \sum_{i=1,2,3} (b_i \cdot \nabla g)^2 \dd \sigma. 
\end{align*}
The second identity follows from Lemma \ref{l:antisymmetric-commutators}.
\end{proof}

The last ingredient needed for the proof of Lemma \ref{l:logpoincare} is a Poincar\'e inequality on the sphere. 
\begin{lemma} \label{l:thirdeigenvalue}
Let $g : S^2 \to \R$. Assume that $g(\sigma) = g(-\sigma)$ and $g$ has average zero on $S^2$. The following inequality holds
\begin{align*} 
\sum_{i = 1,2,3} \int_{S^2} (b_i \cdot \nabla_\sigma g)^2 \dd \sigma \geq 6 \int_{S^2} g^2 \dd \sigma.
\end{align*}
\end{lemma}

\begin{proof}
The identity \eqref{e:decompositionofaij_applied_to_S2} says that $|\nabla_\sigma g|^2 = \sum \limits_{i=1,2,3}(b_i\cdot \nabla_\sigma g)^2$, so the desired inequality is equivalent to
\begin{align*}
\int_{S^2}  |\nabla_\sigma g|^2 \dd \sigma \geq 6 \int_{S^2} g^2 \dd \sigma.
\end{align*}
This inequality amounts to an elementary observation about the eigenvalues and eigenfunctions of $-\Delta_\sigma$, which are well understood. The eigenfunctions of $-\Delta_\sigma$ correspond to spherical harmonics. The first three eigenvalues are $0$, $2$ and $6$. They correspond to constant functions, first-order spherical harmonics, and second-order spherical harmonics. Since $g$ is average zero and $g(\sigma)=g(-\sigma)$, then it is orthogonal in $L^2(S^2)$ to the eigenspaces corresponding to the first and second eigenfunctions. We obtain the inequality of the lemma because the third eigenvalue of $-\Delta_\sigma$ equals $6$.
\end{proof}

\begin{lemma} \label{l:symmetricinequality}
Let $f:S^2 \to (0,\infty)$ be $C^2$ and even (i.e. $f(\sigma)=f(-\sigma)$), and $M_g$ as in \eqref{e:MN}. The following inequality holds.
\[ \int_{S^2} \|M_{\log f}(\sigma)\|^2 \ f \dd \sigma \geq \frac{17}4 \sum_{i=1,2,3} \int_{S^2} |b_i \cdot \nabla \log f|^2 \ f \dd \sigma.\]
\end{lemma}

\begin{proof}
We start from applying Lemma \ref{l:lp2}.
\begin{align*}
\int_{S^2} \|M_{\log f}(\sigma)\|^2 \ f \dd \sigma &\geq \int_{S^2} 2\|M_{\sqrt{f}}(\sigma)\|^2 + |\trace M_{\sqrt{f}}(\sigma)|^2 \dd \sigma \\
&= \int_{S^2} 2 \left( \sum_{i,j=1,2,3} (b_i \cdot \nabla (b_j \cdot \nabla \sqrt f))^2 \right) - 2 \| N_{\sqrt{f}}(\sigma)\|^2 + |\trace M_{\sqrt{f}}(\sigma)|^2 \dd \sigma. \\
\intertext{We use Lemma \ref{l:antisymmetric-commutators} to compute the middle term that involves $N_{\sqrt f}$,}
&= \int_{S^2} 2 \left( \sum_{i,j=1,2,3} (b_i \cdot \nabla (b_j \cdot \nabla \sqrt f))^2 \right) + |\trace M_{\sqrt{f}}(\sigma)|^2 \dd \sigma - \frac 14 \int_{S^2} \sum_{i=1,2,3} (b_i \cdot \nabla \log f)^2 f \dd \sigma. \\
\intertext{We apply Lemma \ref{l:laplaciansquared} to $g = \sqrt f$ to replace the second term,}
&= \int_{S^2} 3 \left( \sum_{i,j=1,2,3} (b_i \cdot \nabla (b_j \cdot \nabla \sqrt f))^2 \right) - \frac 14 \int_{S^2} \sum_{i=1,2,3} (b_i \cdot \nabla \log f)^2 f \dd \sigma. \\
\intertext{Applying Lemma \ref{l:thirdeigenvalue} to $g = b_j \cdot \nabla \sqrt f$,}
&\geq 18 \int_{S^2} \sum_{j=1,2,3}(b_j \cdot \nabla \sqrt f)^2 \dd \sigma - \frac 14 \int_{S^2} \sum_{i=1,2,3} (b_i \cdot \nabla \log f)^2 f \dd \sigma \\
&\geq (\frac{18}4 - \frac 14) \int_{S^2} \sum_{i=1,2,3} (b_i \cdot \nabla \log f)^2 f \dd \sigma.
\end{align*}
\end{proof}

\begin{proof} [Proof of Lemma \ref{l:logpoincare}]
We combine Lemma \ref{l:antisymmetric-commutators} with Lemma \ref{l:symmetricinequality}.
\end{proof}

\begin{remark} \label{r:intrinsic}
The quantities involved in this section correspond to intrinsic geometric objects on the sphere $S^2$. The identity \ref{e:laplacebeltrami} says that for any $C^2$ function $g : S^2 \to \R$,
\[ \sum_{i=1,2,3} b_i \cdot \nabla(b_i \cdot \nabla g) = \Delta_\sigma g.\]
It can be seen that other quantities that play a role in this section correspond to the following objects.
\begin{align*}
\|M_g\|^2 &= \|\nabla^2_\sigma g\|^2 + \frac 12 |\nabla_\sigma g|^2, \\
\|N_g\|^2 &= \frac{1}{2}\sum_{i=1,2,3} |b_i \cdot \nabla_\sigma g|^2 = \frac{1}{2}|\nabla_\sigma g|^2.
\end{align*}
Here $\nabla_\sigma$ and $\Delta_\sigma$ are respectively the gradient and the Laplace-Beltrami operators in $S^2$ with its standard metric. The proofs in this section can be written in terms of these intrinsic objects. We write it in terms of the vector fields $b_i$ to keep a uniform notation throughout the paper, and to keep the proof more elementary.

Moreover, the integrands in Lemma \ref{l:logpoincare} correspond to the \emph{carre-du-champ} operator $\Gamma$ and the operator $\Gamma_2$ in the Bakry-Emery formalism. In fact, from the definition of $\Gamma$ and $\Gamma_2$ one can check for any smooth function $g:S^2\to\R$
\begin{align*}
 \sum_{i,j = 1,2,3} (b_i \cdot \nabla g)^2 &= \Gamma(g,g), \\
 \sum_{i,j = 1,2,3} (b_i \cdot \nabla (b_j \cdot \nabla g))^2 &= \Gamma_2(g,g).
\end{align*}
One only needs to make use of the above expressions for $\nabla_\sigma g$ and $\Delta_\sigma g $ in terms of the $b_i$'s, expand the resulting terms, and use the commutator identities for all $[b_i,b_j]$ to get the needed cancellations. For the reader's convenience we recall the definition of the operators $\Gamma$ and $\Gamma_2$ associated to the Laplacian $\Delta_\sigma$,
\begin{align*}
\Gamma(f,g) & := \tfrac{1}{2}\left(\Delta_\sigma (fg)-f (\Delta_\sigma g)-(\Delta_\sigma f)g \right)\\
\Gamma_2(f,g) & := \tfrac{1}{2}\left(\Delta_\sigma (\Gamma(f,g))-\Gamma(f,\Delta_\sigma g)-\Gamma(\Delta_\sigma f,g)\right).
\end{align*}	
\end{remark}

\begin{remark} \label{r:Gamma2}
The optimal constant for the $\Gamma_2$ criterion on the sphere is well-known to be equal to $2$ (see \cite[Section 5.7]{bakry-2014}), which would not suffice to prove our main result in the case of Coulomb potentials. We get an advantage due to the symmetry assumption $f(\sigma) = f(-\sigma)$ that is effectively equivalent to consider functions on the projective space $\mathbb RP^2$. A discussion of the best known constants in this setting, and an improvement over the constant of Lemma \ref{l:logpoincare} is given in \cite{sehyun2024}.
\end{remark}

\section{The monotonicity of the Fisher information}

In this section we complete the proof of Theorem \ref{t:main} as a consequence of the results in the previous sections.

\begin{lemma} \label{l:controlofbadterm}
Let $F: \R^6 \to (0,\infty)$ be a smooth function with rapid decay at infinity so that $F(v,w) = F(w,v)$. Let $D_{spherical}$ and $R_{spherical}$ be the quantities defined in \eqref{e:Dspherical} and \eqref{e:Rspherical} respectively. The following inequality holds,
\[ D_{spherical} \geq \frac{19}2 R_{spherical}.\]
\end{lemma}

\begin{proof}
Let us parametrize $v$ and $w$ in the following way. We let $z := (v+w)/2 \in \R^3$ and we write $(v-w)/2 = r \sigma$ for $r \in [0,\infty)$ and $\sigma \in S^2$. In these coordinates,
\begin{equation} \label{e:zrsigma}
\begin{aligned}
v &= z + r \sigma \\
w &= z - r \sigma.
\end{aligned}
\end{equation}

The Jacobian of this change of variables corresponds to
\[ \dd v \dd w = 8 r^2 \dd \sigma \dd z \dd r. \]

Note that $|v-w| = 2r$ and $b_k(v-w) = 2r b_k(\sigma)$.  Moreover, for any vector $b$ perpendicular to $\sigma$, we can think of $b$ as a vector on the tangent space of $S^2$ at $\sigma$ and we observe that for any function $G: \R^6 \to \R$, $b \cdot \nabla_\sigma G = r b \cdot (\nabla_v - \nabla_w) G$. In particular, with $G = \log F$, it leads to
\[ \tilde b_k(v-w) \cdot \nabla \log F = 2r b_k(\sigma) \cdot (\nabla_v - \nabla_w) \log F =  2 b_k(\sigma) \cdot \nabla_\sigma \log F.\]

The symmetry condition $F(v,w) = F(w,v)$ translates to $F(z,r,\sigma) = F(z,r,-\sigma)$ in terms of the new variables.

We rewrite $D_{spherical}$ and $R_{spherical}$ in these coordinates, 
\begin{align*}
D_{spherical} &= \sum_{i,j=1,2,3} \int_{\R^3} \int_0^\infty \int_{S^2} \frac {\alpha(2r)}{8r^2} \ 16 | b_i(\sigma) \cdot \nabla_\sigma ( b_j(\sigma) \cdot \nabla_\sigma \log F) |^2 \ F  \ (8r^2 \dd \sigma \dd r \dd z) \\
&= 16 \int_{\R^3} \int_0^\infty \alpha(2r) \left( \sum_{i,j=1,2,3} \int_{S^2} | b_i(\sigma) \cdot \nabla_\sigma ( b_j(\sigma) \cdot \nabla_\sigma \log F) |^2 \ F \dd \sigma \right) \dd r \dd z, \\
R_{spherical} &= \sum_{k=1,2,3} \int_{\R^3} \int_0^\infty \int_{S^2} \frac{\alpha(2r)}{4 r^2} \ 4(b_k(\sigma) \cdot \nabla_\sigma \log F)^2 \ F \ (8r^2 \dd \sigma \dd r \dd z) \\
&= 8 \int_{\R^3} \int_0^\infty \alpha(2r) \left( \sum_{k=1,2,3} \int_{S^2} (b_k(\sigma) \cdot \nabla_\sigma \log F)^2 \ F \dd \sigma \right) \dd r \dd z.
\end{align*}

For each value of $z \in \R^3$ and $r \in (0,\infty)$, the inner integrals with respect to $\sigma$ (the ones inside the parenthesis) satisfy the inequality of Lemma \ref{l:logpoincare}. Therefore, we conclude that $\frac 12 D_{spherical} \geq \frac{19}4 \ R_{spherical}$. Equivalently, $D_{spherical} \geq \frac{19}2 R_{spherical}$.
\end{proof}

\begin{prop} \label{p:finalnail}
If $\alpha(r) \geq 0$ is any interaction potential so that for all $r > 0$,
\[ \frac{r |\alpha'(r)|}{\alpha(r)} \leq \sqrt{19},\]
then \eqref{e:question} holds for any smooth function $F: \R^6 \to (0,\infty)$ with rapid decay at infinity so that $F(v,w) = F(w,v)$.
\end{prop}

\begin{proof}
Combine Lemma \ref{l:newremainder} with Lemma \ref{l:controlofbadterm}.
\end{proof}

Theorem \ref{t:main} is essentially proved already.

\begin{proof}[Proof of Theorem \ref{t:main}]
Proposition \ref{p:finalnail} tells us that \eqref{e:question} holds whenever $f:\R^3 \to (0,\infty)$ is smooth, well behaved at infinity and strictly positive. Therefore, the right hand side in Lemma \ref{l:dti=dtI/2} is nonpositive and Theorem \ref{t:main} follows in that case.

If $f$ has vacuum regions, or if its tails are not sufficiently well behaved, we can approximate $F$ with a smooth and well behaved function and pass to the limit. For example, let $\eta : \R^6 \to [0,1]$ be a smooth function so that $\eta \equiv 1$ in $B_1$ and $\eta \equiv 0$ in $\R^6 \setminus B_2$. We set
\[ F_\eps(v,w) = (F(v,w) + \eps) \eta\left(\eps v, \eps w \right) + \left( 1- \eta\left(\eps v, \eps w \right) \right) \exp(-|v|^2). \]
This function $F_\eps$ converges to $F$ as $\eps > 0$. Moreover, for each $\eps>0$, $F_\eps > 0$ and equals a Maxwellian for large values of $v$ and $w$. We thus know that \eqref{e:question} holds for $F_\eps$ and then we deduce it also holds for $F$. For a more thorough description of this technical approximation argument, see Appendix \ref{a:byparts}.
\end{proof}

\begin{remark} \label{r:sphere}
The operator $Q$ has a very simple expression in terms of the variables $z$, $r$ and $\sigma$. From the observation that $\tilde b_i \cdot \nabla F = 2 b_i(\sigma) \cdot \nabla_\sigma F$ and \eqref{e:laplacebeltrami}, we see that
\[ Q(F) = 4\alpha(2r) \ \Delta_\sigma F.\]
Here, $\Delta_\sigma$ is the Laplace-Beltrami operator with respect to $\sigma$ on $S^2$.
\end{remark}

\section{The global existence theorem}
\label{s:globalexistence}

In this section we explain how Theorem \ref{t:main} is used to obtain Theorem \ref{t:main2}.

\begin{proof} [Proof of Theorem \ref{t:main2}]
We recall that the result is already well known in the case $\gamma \in [0,1]$ from \cite{villani1998spatiallyhomogeneous,desvillettes2000}. It is also easy to derive for $\gamma \in [-2,0]$ as a consequence of the upper bound in \cite{silvestre2017}. We focus on the case $\gamma<-2$.

Given any initial data as in Theorem \ref{t:main2}, we construct a solution for a short period of time $[0,T)$, with $T>0$, using Theorem \ref{t:henderson-snelson-tarfulea}. This solution becomes immediately smooth and rapidly decaying. Applying Lemma \ref{l:finite-fisher}, we observe that the Fisher information $i(f)$ also becomes finite for any small $t>0$.

The continuation criteria in Theorem \ref{t:henderson-snelson-tarfulea} tells us that the solution can be extended for as long as $\|f(t)\|_{L^\infty_{k_0}}$ is bounded. In other words, the solution may blow up at time $T$ only if $\lim_{t \to T} \|f\|_{L^\infty_{k_0}} = +\infty$.

Also from Theorem \ref{t:henderson-snelson-tarfulea}, during the interval $[0,T)$, the solution is $C^\infty$, strictly positive, and bounded by $C(t) \langle v \rangle^{-k_0}$ (here $C(t) = \|f(t)\|_{L^\infty_{k_0}}$ might be blowing up as $t \to T$). Under these conditions, we are able to apply Theorem \ref{t:main} and deduce that its Fisher information is monotone decreasing on the time interval $(0,T)$. In particular, for $t_0>0$ small, the Fisher information will remain smaller or equal to $i(f(t_0))$ in the interval $[t_0,T]$. It implies the uniform boundedness of $\|f(t)\|_{L^3}$. Indeed,
\begin{align*}
\|f\|_{L^3} = \|\sqrt f\|_{L^6}^2 \leq C \|\sqrt f\|_{\dot H^1}^2 = C i(f)/4.
\end{align*}
From Theorem \ref{t:moments}, we also know that $\|f(t)\|_{L^1_q}$ remains bounded in $[0,T]$ for all $q > 0$. Interpolating between $\|f(t)\|_{L^1_q}$ and $\|f\|_{L^3}$ we deduce that $\|f(t)\|_{L^p_k}$ is bounded for any $p$ in the range $(1,3)$ and any large exponent $k$. We can thus apply Theorem $\ref{t:silvestre2017}$ since $3/(5+\gamma) < 3$. We deduce that
\[ \|f(t)\|_{L^\infty} \leq C_3 (1+(t-t_0)^{-3/(2p)}),\]
for a constant $C_3$ that depends only on $p$ and the mass, energy and Fisher information of $f_0$. Since the function $f$ is certainly bounded in some short time interval $[0,\delta]$ (from Theorem \ref{t:henderson-snelson-tarfulea}), we deduce that
\begin{equation} \label{e:e1} 
\|f(t)\|_{L^\infty} \leq C_4 \text{ for } t \in [0,T),
\end{equation}
for a constant $C_4$.

We want to apply Theorem \ref{t:cameron2018}. It propagates a Maxwellian upper bound when the exponent $\beta>0$ is sufficiently small depending on the mass, energy and entropy of $f$. This is not a restriction for us, because if $f_0(v) \leq C_0 \exp(-\beta |v|^2)$ for some $C_0$ and $\beta>0$, then the same inequality also holds with a smaller value of $\beta$. We can therefore assume without loss of generality that $\beta>0$ is small.

Using the upper bound \eqref{e:e1} and the moment bounds together with Theorem \ref{t:cameron2018}, we deduce uniform Maxwellian upper bounds for $t \in [0,T)$
\[ f(t,v) \leq C_5 \exp(-\beta |v|^2). \]
 But this means that $\|f\|_{L^\infty_k}$ is uniformly bounded in $[0,T)$ for any exponent $k$. Thus, the solution can never blow up according to the continuation criteria in Theorem \ref{t:henderson-snelson-tarfulea}.

 The uniqueness of the solution follows from Theorem \ref{t:uniqueness}.
\end{proof}
 
\begin{remark}
Note that all the only a priori estimate on $[0,T]$ used in the proof of Theorem \ref{t:main2} that potentially deteriorates as $T \to \infty$ is the moment estimate from Theorem \ref{t:moments} when $\gamma \leq 0$. 
\end{remark}

\appendix\section{the two-dimensional case}

It is slightly simpler to prove a version of Theorem \ref{t:main} in two dimensions, following approximately the same steps in the proof. We state the result here.

\begin{thm} \label{t:2d}
Let $f : [0,T] \times \R^2 \to [0,\infty)$ be a classical solution to the space-homogeneous Landau equation \eqref{e:landauequation}. Assume that the interaction potential $\alpha$ satisfies, for all $r>0$,
\[ \frac{r |\alpha'(r)|}{\alpha(r)} \leq 4,\]
then the Fisher information $i(f)$ is monotone decreasing as a function of time.
\end{thm}

The proof of Theorem \ref{t:2d} follows the same steps as the proof of Theorem \ref{t:main}. We sketch the differences here. We only need one vector field $b_1$ to write $Q(F)$, for a function $F:\R^4 \to [0,\infty)$. We write 
\[ b_1(v-w) = \begin{pmatrix} v_2-w_2 \\ w_1 - v_1 \end{pmatrix}, \qquad \tilde b_1 = \begin{pmatrix} b_1 \\ -b_1 \end{pmatrix}.\]
With this notation, we have
\[ Q(F) = \alpha \tilde b_1 \cdot \nabla( \tilde b_1 \cdot \nabla F).\]
Since we have only one vector $b_1$ instead of the three vectors $b_1$, $b_2$ and $b_3$, the resulting formulas are simpler and involve no summation. Following the same line of thought as for the 3D case we end up with
\[ \frac 12 \langle I'(F) Q(F), Q(F) \rangle \leq -D_{parallel} -D_{radial} - D_{spherical} + \sup_{r>0} \left( \frac{r^2 \alpha'(r)^2}{2 \alpha(r)^2}\right) R_{spherical}, \]
where
\begin{align}
D_{parallel} &:= \frac 12 \int \alpha(|v-w|) F |(\partial_{v_i}+\partial_{w_i}) \tilde b_1 \cdot \nabla \log F|^2 \dd w \dd v \\
D_{radial} &:= \int F |a \cdot \nabla \left( \sqrt \alpha \ \tilde b_1 \cdot \nabla \log F \right)|^2 \dd w \dd v \\
D_{spherical} &:= \int \frac \alpha{2|v-w|^2} F | \tilde b_1 \cdot \nabla ( \tilde b_1 \cdot \nabla \log F) |^2 \dd w \dd v \\
R_{spherical} &:= \int \frac{\alpha}{|v-w|^2} \ F \ (\tilde b_1 \cdot \nabla \log F)^2 \dd w \dd v.
\end{align}

We still want to control $R_{spherical}$ with $D_{spherical}$. It is achieved after the following lemma.
\begin{lemma} \label{l:logpoincare2d}
Let $f : S^1 \to (0,\infty)$ be a $C^2$ function on the circle. Assume that $f(\sigma) = f(-\sigma)$ for all $\sigma \in S^1$. The following inequality holds
\begin{align*} 
\int_{S^1} f (b_1 \cdot \nabla (b_1 \cdot \nabla \log f))^2 \dd \sigma \geq 4 \int_{S^1} f (b_1 \cdot \nabla \log f)^2 \dd \sigma.
\end{align*}
\end{lemma}

In this case the factor $4$ on the right-hand side is optimal. It is achieved asymptotically for a function $f = 1 + \eps g$ as $\eps \to 0$ when $g$ is the third eigenfunction of the Laplacian on $S^1$.

The proof of Lemma \ref{l:logpoincare2d} is significantly easier than Lemma \ref{l:logpoincare}. Since there is only one direction $b_1$, we apply Lemma \ref{l:singledirectioninequality} and obtain right away
\begin{align*}
\int_{S^1} f (b_1 \cdot \nabla (b_1 \cdot \nabla \log f))^2 \dd \sigma &\geq 4 \int_{S^1} (b_1 \cdot \nabla (b_1 \cdot \nabla \sqrt f))^2 \dd \sigma \\
\intertext{Since $f(\sigma)=f(-\sigma)$, then $b_1 \cdot \nabla f$ is orthogonal to the first two eigenspaces. The third eigenvalues of $-\partial_{\sigma}^2$ in $S^1$ is equal to four.}
&\geq 16 \int_{S^1} (b_1 \cdot \nabla \sqrt f)^2 \dd \sigma = 4 \int_{S^1} f (b_1 \cdot \nabla \log f)^2 \dd \sigma.
\end{align*}
We conclude that $D_{spherical} \geq 8 R_{spherical}$, and use it to finish the proof of Theorem \ref{t:2d}.

\medskip

If we want to carry out the analysis in this paper in $\R^d$ for $d>3$, we would have to consider $d(d-1)/2$ vectors $b_i$'s. For example, in four dimensions, the six vectors would be
\begin{align*} 
b_1 &= \begin{pmatrix} v_2-w_2 \\ w_1 - v_1 \\ 0 \\ 0 \end{pmatrix}, \ b_2 = \begin{pmatrix} w_3-v_3 \\ 0 \\ v_1-w_1 \\ 0 \end{pmatrix}, \ b_3 = \begin{pmatrix} v_4-w_4 \\ 0 \\ 0 \\ w_1 - v_2 \end{pmatrix},\\
\phantom{.}\\
b_4 &= \begin{pmatrix} 0 \\ v_3-w_3 \\ w_2 - v_2 \\ 0 \end{pmatrix}, \ b_5 = \begin{pmatrix} 0 \\ w_4-v_4 \\ 0 \\ v_2 - w_2 \end{pmatrix}, \ b_6 = \begin{pmatrix} 0 \\ 0 \\ v_4-w_4 \\ w_3 - v_3 \end{pmatrix}.
\end{align*}

There is a result similar to Lemma \ref{l:logpoincare} in any dimension. However, the constant factor on the right-hand side depends on the dimension, as well as the range of admissible values of $|r\alpha'(r)/\alpha(r)|$ for the higher dimensional counterpart of Theorem \ref{t:main}. The values we computed in this paper for the three dimensional case are probably not optimal. It would require some work to compute the sharp range in arbitrary dimension.

\section{On the decay of the tails of our integrals} 
\label{a:byparts}


Throughout this paper, we work with solutions $f$ that are $C^\infty$ and decay as $|v| \to \infty$ faster than any algebraic rate. The equation \eqref{e:landauequation} is understood in the classical sense. There are several instances where we consider the derivative of the Fisher information and we end up with integrals in $\R^6$ involving two or three derivatives of $\log f$. All the integrands are homogeneous of degree one on $f$. It is natural to expect them to decay rapidly as $|v|\to \infty$. However, some justification is required since $\log f$ and its derivatives have some growth as $|v|\to \infty$. It is not completely inappropriate (even if admittedly pedantic) to provide a justification that all the integrals in this paper make sense. We describe it in this appendix so that the reader is not distracted through the main text of the article.

In the proof of Theorem \ref{t:main}, we argued that a generic function $F$ can be approximated with a strictly positive function which agrees with a Maxwellian for large velocities. This approximation can be used, whenever necessary, to justify that the lemmas and inequalities throughout this paper apply to a much wider class of functions. In this appendix, we show that the upper and lower Maxwellian bounds introduced in this approximation are propagated in time by the Landau equation. Thus, we show that these solutions to the Landau equation \eqref{e:landauequation} will always be well behaved, $C^\infty$ smooth, strictly positive, and with well behaved tails for $|v| \to \infty$. Theorem \ref{t:henderson-snelson-tarfulea} justifies the majority of these statements. The only condition that remains to be justified is that the derivatives of $\log F$ are appropriately bounded for large velocities so that the tails of the integrals throughout this paper are convergent. In this appendix, we describe the procedure to approximate the whole solution to \eqref{e:landauequation} with solutions $f^\eps$ for which we verify these bounds.


One sufficient condition that would easily validate all the integral expressions in this article would be when $f \lesssim \exp(-\beta|v|^2)$ and $(1+|\nabla \log f|+|D^2 \log f|) \lesssim \exp(\eps |v|^2)$ for $\eps \ll \beta$. We will show that this condition is satisfied for a general class of initial data $f_0$. If $f_0$ is bounded below and above by a multiple of the same Maxwellian $f_0 \approx \exp(-\beta |v|^2)$, then these bounds can be propagated in time following the ideas in \cite{cameron2018} for as long as there is a classical solution to the equation. Bounds on the derivatives of $\log f$ follow applying standard parabolic estimates.

Below, we briefly review the propagation of Gaussian bounds following techniques from the literature. It is a completely standard technique, so we only sketch the proofs here.

We explain the propagation of Gaussian bounds in the case $\alpha(r) = r^\gamma$ with $\gamma \in [-3,0]$. In the case $\gamma \in [0,1]$, upper and lower bounds of the same kind are obtained in \cite{desvillettes2000}.

\begin{prop} \label{p:lowerbound}
Let $f : [0,T] \times \R^3 \to [0,\infty)$ be a solution of \eqref{e:landauequation} with $\alpha(r) = r^\gamma$, $\gamma \in [-3,0]$ and initial data $f_0$. Assume that for some $\beta>0$, $\delta_0>0$ and $C_0$,
\begin{align*}
f_0(v) \geq \delta_0 \exp(-\beta |v|^2), \\
f(t,v) \leq C_0 \exp(-\beta |v|^2).
\end{align*}
Then, for any $\eps>0$, there is a $\delta_1>0$ so that
\[ f(t,v) \geq \delta_1 \exp(-(\beta+\eps)|v|^2),\]
for all $t \in [0,T]$ and $v \in \R^3$.
\end{prop}

\begin{proof}[Sketch proof]
We write the Landau collision operator in non-divergence form
\[ q(f) = \bar a_{ij} \partial f + \bar c f,\]
where 
\[ \bar a_{ij} = \int \alpha(|v-w|) a_{ij}(v-w) f(w) \dd w, \qquad \bar c = -\partial_{ij} \bar a_{ij} \geq 0.\]

We follow the same idea as in \cite[Theorem 4.3]{cameron2018} but with a Maxwellian bound from below. We must find a function $\psi(t,v)$ which is a subsolution to
\[ \partial_t \psi \leq \bar a_{ij} \partial_{ij} \psi + \bar c \psi. \]
Following \cite{cameron2018}, we know that if $\psi(0,v) \leq f(0,v)$, then we will also have $\psi(t,v) \leq f(t,v)$ for all $t \in [0,T]$ and $v \in \R^3$. We claim that the function $\psi(t,v) = \delta_0 \exp(-C_1 t - \eps t |v|^2) \exp(-\beta |v|^2)$ satisfies this differential inequality.

It is well known that the coefficients $\bar a_{ij}$ satisfy certain ellipticity bounds. We have, for some constants $\Lambda \geq \lambda > 0$,
\begin{equation} \label{e:ellipticity}
\lambda \langle v \rangle^\gamma \left(  \left(|v|^2 I - (v \otimes v)\right) + \langle v \rangle^{-2} (v \otimes v) \right) \leq \{ \bar a_{ij} \} \leq \Lambda \langle v \rangle^\gamma \left( \left(|v|^2 I - (v \otimes v)\right) + \langle v \rangle^{-2} (v \otimes v) \right).
\end{equation}
The constant $\lambda>0$ in the lower bound depends only on the mass, energy and entropy of $f_0$. The proof is the same for any value of $\gamma \in [-3,1]$. It can be found in \cite[Lemma 3.1]{silvestre2017} and \cite[Proposition 4]{desvillettes2000}. For a proof of the upper bound, see \cite[Lemma 2.1]{cameron2018} and the proof of Theorem \ref{t:cameron2018} in Section \ref{s:preliminaries}.

Plugging these estimates on our function $\psi$, we see that
\begin{align*}
\partial_t \psi &= -(C_1 + \eps |v|^2) \psi, \\
\bar a_{ij} \partial_{ij} \psi + \bar c \psi &\geq \bar a_{ij} \partial_{ij} \psi \\
&\gtrsim -\Lambda \langle v \rangle^{\gamma+2} \psi.
\end{align*}
We pick $C_1$ sufficiently large so that for all $v \in \R^3$ we have $C_1 + \eps |v|^2 \gtrsim \Lambda \langle v \rangle^{\gamma+2}$ and we finish the proof.
\end{proof}

The derivatives of $f$ can be bounded using standard parabolic estimates.

\begin{prop} \label{p:schauder}
Let $f : [0,T] \times \R^3 \to [0,\infty)$ be a solution of \eqref{e:landauequation} with $\alpha(r) = r^\gamma$, $\gamma \in [-3,0]$. Assume that for some $\beta>0$ and $C_0$,
\begin{align*}
f(t,v) \leq C_0 \exp(-\beta |v|^2).
\end{align*}
Then, for any $\eps>0$, there is a constant $C_1$ (depending on $T$ and $C_0$) so that
\begin{align*} 
|\nabla_vf(t,v)| &\leq C_1 \exp(-(\beta-\eps)|v|^2), \\
|D^2_vf(t,v)| &\leq C_1 \exp(-(\beta-\eps)|v|^2).
\end{align*}
for all $t \in [T/2,T]$ and $v \in \R^3$.
\end{prop}

\begin{proof}[Sketch proof]
The function $f$ satisfies the equation
\[ f_t = \bar a_{ij} \partial_{ij} f + \bar c f.\]
The coefficients $a_{ij}$ satisfy the ellipticity bounds \eqref{e:ellipticity}. Moreover, from the Gaussian upper bound on $f$, we can deduce that both $\bar a_{ij}$ and $\bar c$ are H\"older continuous in $v$. In order to overcome the difficulty that the ellipticity condition degenerates as $|v| \to \infty$, for every $t_0>0$ and $v_0 \in \R^3$ we use the change of variables $T_{v_0}$ described in \cite[Section 4]{cameron2018}. It maps a parabolic ellipsoid around $(t_0,v_0)$ into $(-1,0] \times B_1$ and the function $f$ into a function $\tilde f$ that satisfies a linear parabolic equation whose coefficients are elliptic with parameters uniform with respect to $v_0$. We may further rescale it to make the H\"older norm of the coefficients less than one. Applying the Schauder estimates to this function $\tilde f$, we obtain
\[ |D^2_v \tilde f(0)| \leq C \|\tilde f\|_{L^\infty((-1,0] \times B_1)} \lesssim \exp(-\beta |v_0|^2).\]
Rewriting the estimate above in terms of the original function $f$, we obtain, for some computable exponent $m \in \mathbb N$,
\[|D^2_v f(t_0,v_0)| \leq C \langle v_0 \rangle^m \exp(-\beta |v_0|^2) \leq C_1 \exp(-(\beta-\eps) |v_0|^2).\]
Some early references to the Schauder estimates depending on the H\"older norm of the coefficients in space only are \cite{brandt1969,knerr1980}.
\end{proof}

Combining the upper bound from Theorem \ref{t:cameron2018}, the lower bound of Proposition \ref{p:lowerbound} and the bound for the derivatives of Proposition \ref{p:schauder}, we see that whenever $f$ is a solution of \eqref{e:landauequation} with initial data $f_0$ so that $\delta_0 \exp(-\beta |v|^2) \leq f_0(v) \leq C_0 \exp(-\beta |v|^2)$, then for any $\eps>0$ and $t>0$ there exists a constant $C_2$ so that\begin{align*}
|\nabla_v \log f(t,v)| &\leq C_2 \exp(\eps |v|^2), \\
|D^2_v \log f(t,v)| &\leq C_2 \exp(\eps |v|^2).
\end{align*}

Not only do the solutions of the Landau equation decay at large velocities but also their derivatives. We can use that to deduce that the coefficients $\bar a_{ij}$ are also $C^\infty$ and deduce estimates like those of Proposition \ref{p:schauder} for higher order derivatives of $f$.

The potential growth of $\nabla \log f$ and $D^2 \log f$ is an order of magnitude smaller than the decay of $f$ as $|v| \to \infty$. This is enough to conclude that every integrand considered in this paper decays faster than some Gaussian rate as $|v| \to \infty$. It shows that all our integrals are well defined and our manipulations (such as integration by parts) are fully justified at least when the initial data $f_0$ satisfies these Gaussian bounds.

If $f_0$ is a generic function with arbitrary decay and perhaps some vacuum regions, we may approximate $f_0$ following the same rule as in the proof of Theorem \ref{t:main}. More precisely, let $f^\eps$ be the solution to \eqref{e:landauequation} with initial data
\[ f^\eps(0,v) = (f_0(v) + \eps) \eta\left(\eps v \right) + \left( 1- \eta\left(\eps v \right) \right) \exp(-|v|^2). \]
Here $\eta: \R^3 \to [0,1]$ is a smooth function supported in $B_2$ so that $\eta \equiv 1$ in $B_1$.

As we discussed above, this solution $f^\eps$ propagates in time the Gaussian bounds from above and below. For this approximate solution all our computations are justified and Theorems \ref{t:main} and \ref{t:main2} hold. We deduce that for every $\eps>0$, there exists a global-in-time smooth solution whose Fisher information is monotone decreasing. As $\eps \to 0$, $f^\eps$ converges to $f$, which is the unique (from Theorem \ref{t:uniqueness}) solution to \eqref{e:landauequation} with initial data $f_0$. The convergence holds for example in for $f$ in $L^{\infty,k}$, for any exponent $k$, and in $L^\infty([0,T],H^1(\R^3))$ for $\sqrt{f}$. We conclude that the Fisher information of an arbitrary solution $f$ of \eqref{e:landauequation} is monotone decreasing in time, assuming only that $f_0 \in L^{\infty,k}$ for $k$ as in Theorem \ref{t:henderson-snelson-tarfulea} and $\sqrt f_0 \in H^1$.

\bibliographystyle{plain}
\bibliography{landau}

\begin{thebibliography}{10}

\bibitem{alonso2023solutions}
R.~Alonso, V.~Bagland, L.~Desvillettes, and B.~Lods.
\newblock A~priori estimates for solutions to {L}andau equation under
  {P}rodi-{S}errin like criteria.
\newblock {\em Arch. Ration. Mech. Anal.}, 248(3):Paper No. 42, 63, 2024.

\bibitem{alonso2019}
Ricardo~J. Alonso, V\'{e}ronique Bagland, and Bertrand Lods.
\newblock Uniform estimates on the {F}isher information for solutions to
  {B}oltzmann and {L}andau equations.
\newblock {\em Kinet. Relat. Models}, 12(5):1163--1183, 2019.

\bibitem{peskov1977}
A.~A. {Arsen'ev} and N.~V. Peskov.
\newblock The existence of a generalized solution of {L}andau's equation.
\newblock {\em \v{Z}. Vy\v{c}isl. Mat i Mat. Fiz.}, 17(4):1063--1068, 1096,
  1977.

\bibitem{bakry-2014}
Dominique Bakry, Ivan Gentil, and Michel Ledoux.
\newblock {\em Analysis and geometry of {M}arkov diffusion operators}, volume
  348 of {\em Grundlehren der mathematischen Wissenschaften [Fundamental
  Principles of Mathematical Sciences]}.
\newblock Springer, Cham, 2014.

\bibitem{bedrossian2022}
Jacob Bedrossian, Maria~Pia Gualdani, and Stanley Snelson.
\newblock Non-existence of some approximately self-similar singularities for
  the {L}andau, {V}lasov-{P}oisson-{L}andau, and {B}oltzmann equations.
\newblock {\em Trans. Amer. Math. Soc.}, 375(3):2187--2216, 2022.

\bibitem{bobylev1975}
A.~V. Bobyl\"{e}v.
\newblock The method of the {F}ourier transform in the theory of the
  {B}oltzmann equation for {M}axwell molecules.
\newblock {\em Dokl. Akad. Nauk SSSR}, 225(6):1041--1044, 1975.

\bibitem{bobylev1988}
A.~V. Bobyl\"{e}v.
\newblock The theory of the nonlinear spatially uniform {B}oltzmann equation
  for {M}axwell molecules.
\newblock In {\em Mathematical physics reviews, {V}ol. 7}, volume~7 of {\em
  Soviet Sci. Rev. Sect. C: Math. Phys. Rev.}, pages 111--233. Harwood Academic
  Publ., Chur, 1988.

\bibitem{bobylev2023}
A.~V. Bobylev.
\newblock Radially symmetric models of the {L}andau kinetic equation and high
  energy tails.
\newblock {\em J. Stat. Phys.}, 190(3):Paper No. 48, 24, 2023.

\bibitem{gamba2015}
Alexander Bobylev, Irene Gamba, and Irina Potapenko.
\newblock On some properties of the {L}andau kinetic equation.
\newblock {\em J. Stat. Phys.}, 161(6):1327--1338, 2015.

\bibitem{bolley2007}
F.~Bolley and J.~A. Carrillo.
\newblock Tanaka theorem for inelastic {M}axwell models.
\newblock {\em Comm. Math. Phys.}, 276(2):287--314, 2007.

\bibitem{brandt1969}
A.~Brandt.
\newblock Interior {S}chauder estimates for parabolic differential- (or
  difference-) equations via the maximum principle.
\newblock {\em Israel J. Math.}, 7:254--262, 1969.

\bibitem{cabrera2023regularization}
Rene Cabrera, Maria~Pia Gualdani, and Nestor Guillen.
\newblock Regularization estimates of the landau--coulomb diffusion.
\newblock {\em Nonlinear Analysis}, 251:113695, 2025.

\bibitem{cameron2018}
Stephen Cameron, Luis Silvestre, and Stanley Snelson.
\newblock Global a priori estimates for the inhomogeneous {L}andau equation
  with moderately soft potentials.
\newblock {\em Ann. Inst. H. Poincar\'{e} C Anal. Non Lin\'{e}aire},
  35(3):625--642, 2018.

\bibitem{carlen1991}
Eric~A. Carlen.
\newblock Superadditivity of {F}isher's information and logarithmic {S}obolev
  inequalities.
\newblock {\em J. Funct. Anal.}, 101(1):194--211, 1991.

\bibitem{desvillettes2015}
L.~Desvillettes.
\newblock Entropy dissipation estimates for the {L}andau equation in the
  {C}oulomb case and applications.
\newblock {\em J. Funct. Anal.}, 269(5):1359--1403, 2015.

\bibitem{amelie2024}
Laurent Desvillettes, William Golding, Maria~Pia Gualdani, and Amelie Loher.
\newblock Production of the {F}isher information for the {L}andau-{C}oulomb
  equation with {$L^1$} initial data.
\newblock {\em arXiv preprint arXiv:2410.10765}, 2024.

\bibitem{desvillettes2020new}
Laurent Desvillettes, Ling-Bing He, and Jin-Cheng Jiang.
\newblock A new monotonicity formula for the spatially homogeneous landau
  equation with coulomb potential and its applications.
\newblock {\em J. Eur. Math. Soc}, 2023.

\bibitem{desvillettes2000}
Laurent Desvillettes and C\'{e}dric Villani.
\newblock On the spatially homogeneous {L}andau equation for hard potentials.
  {I}. {E}xistence, uniqueness and smoothness.
\newblock {\em Comm. Partial Differential Equations}, 25(1-2):179--259, 2000.

\bibitem{desvillattes2000II}
Laurent Desvillettes and C\'{e}dric Villani.
\newblock On the spatially homogeneous {L}andau equation for hard potentials.
  {II}. {$H$}-theorem and applications.
\newblock {\em Comm. Partial Differential Equations}, 25(1-2):261--298, 2000.

\bibitem{elsafadi2007}
Mouhamad El~Safadi.
\newblock Smoothness of weak solutions of the spatially homogeneous {L}andau
  equation.
\newblock {\em Anal. Appl. (Singap.)}, 5(1):29--49, 2007.

\bibitem{fournier2010uniqueness}
Nicolas Fournier.
\newblock Uniqueness of bounded solutions for the homogeneous {L}andau equation
  with a {C}oulomb potential.
\newblock {\em Comm. Math. Phys.}, 299(3):765--782, 2010.

\bibitem{fournier2009}
Nicolas Fournier and H\'{e}l\`ene Gu\'{e}rin.
\newblock Well-posedness of the spatially homogeneous {L}andau equation for
  soft potentials.
\newblock {\em J. Funct. Anal.}, 256(8):2542--2560, 2009.

\bibitem{golding2023global}
William Golding, Maria~Pia Gualdani, and Am{\'e}lie Loher.
\newblock Nonlinear regularization estimates and global well-posedness for the
  landau--coulomb equation near equilibrium.
\newblock {\em SIAM Journal on Mathematical Analysis}, 56(6):8037--8069, 2024.

\bibitem{golding2023local}
William Golding and Am{\'e}lie Loher.
\newblock Local-in-time strong solutions of the homogeneous landau--coulomb
  equation with {$L^p$} initial datum.
\newblock {\em La Matematica}, 3(1):337--369, 2024.

\bibitem{golse2021}
Fran\c{c}ois Golse.
\newblock Partial regularity in time for the {L}andau equation (with {C}oulomb
  interaction).
\newblock In {\em From particle systems to partial differential equations},
  volume 352 of {\em Springer Proc. Math. Stat.}, pages 283--300. Springer,
  Cham, [2021] \copyright 2021.

\bibitem{golse2022asens}
Fran\c{c}ois Golse, Maria~Pia Gualdani, Cyril Imbert, and Alexis Vasseur.
\newblock Partial regularity in time for the space-homogeneous {L}andau
  equation with {C}oulomb potential.
\newblock {\em Ann. Sci. \'{E}c. Norm. Sup\'{e}r. (4)}, 55(6):1575--1611, 2022.

\bibitem{golse2022local}
Fran{\c{c}}ois Golse, Cyril Imbert, Sehyun Ji, and Alexis~F Vasseur.
\newblock Local regularity for the space-homogenous landau equation with very
  soft potentials.
\newblock {\em Journal of Evolution Equations}, 24(4):1--81, 2024.

\bibitem{gressman2012}
Philip~T. Gressman, Joachim Krieger, and Robert~M. Strain.
\newblock A non-local inequality and global existence.
\newblock {\em Adv. Math.}, 230(2):642--648, 2012.

\bibitem{gualdani2019Ap}
Maria Gualdani and Nestor Guillen.
\newblock On {$A_p$} weights and the {L}andau equation.
\newblock {\em Calc. Var. Partial Differential Equations}, 58(1):Paper No. 17,
  55, 2019.

\bibitem{gualdani2022hardy}
Maria Gualdani and Nestor Guillen.
\newblock Hardy's inequality and the isotropic {L}andau equation.
\newblock {\em J. Funct. Anal.}, 283(6):Paper No. 109559, 25, 2022.

\bibitem{gualdani2016}
Maria~Pia Gualdani and Nestor Guillen.
\newblock Estimates for radial solutions of the homogeneous {L}andau equation
  with {C}oulomb potential.
\newblock {\em Anal. PDE}, 9(8):1772--1809, 2016.

\bibitem{guo2002}
Yan Guo.
\newblock The {L}andau equation in a periodic box.
\newblock {\em Comm. Math. Phys.}, 231(3):391--434, 2002.

\bibitem{henderson2020}
Christopher Henderson, Stanley Snelson, and Andrei Tarfulea.
\newblock Local solutions of the {L}andau equation with rough, slowly decaying
  initial data.
\newblock {\em Ann. Inst. H. Poincar\'{e} C Anal. Non Lin\'{e}aire},
  37(6):1345--1377, 2020.

\bibitem{hwang2020}
Hyung~Ju Hwang and Jin~Woo Jang.
\newblock Compactness properties and local existence of weak solutions to the
  {L}andau equation.
\newblock {\em Proc. Amer. Math. Soc.}, 148(12):5141--5157, 2020.

\bibitem{imbert2024monotonicity}
Cyril Imbert, Luis Silvestre, and C{\'e}dric Villani.
\newblock On the monotonicity of the {F}isher information for the {B}oltzmann
  equation.
\newblock {\em arXiv preprint arXiv:2409.01183}, 2024.

\bibitem{ji2023entropy}
Sehyun Ji.
\newblock Entropy dissipation estimates for the {L}andau equation with coulomb
  potentials.
\newblock {\em arXiv preprint arXiv:2305.09841}, 2023.

\bibitem{sehyun2024}
Sehyun Ji.
\newblock Bounds for the optimal constant of the bakry-\'emery {$\Gamma_2$}
  criterion inequality on {$\mathbb RP^{d-1}$}.
\newblock {\em arXiv preprint 2408.13954}, 2024.

\bibitem{sehyun2024dissipation}
Sehyun Ji.
\newblock Dissipation estimates of the fisher information for the landau
  equation.
\newblock {\em arXiv preprint arXiv:2410.09035}, 2024.

\bibitem{knerr1980}
Barry~F. Knerr.
\newblock Parabolic interior {S}chauder estimates by the maximum principle.
\newblock {\em Arch. Rational Mech. Anal.}, 75(1):51--58, 1980/81.

\bibitem{krieger2012}
Joachim Krieger and Robert~M. Strain.
\newblock Global solutions to a non-local diffusion equation with quadratic
  non-linearity.
\newblock {\em Comm. Partial Differential Equations}, 37(4):647--689, 2012.

\bibitem{landau1936}
Lev~Davidovich Landau.
\newblock Kinetic equation for the case of {C}oulomb interaction.
\newblock Technical Report~10, 1936.

\bibitem{matthes2012}
Daniel Matthes and Giuseppe Toscani.
\newblock Variation on a theme by {B}obyl\"{e}v and {V}illani.
\newblock {\em C. R. Math. Acad. Sci. Paris}, 350(1-2):107--110, 2012.

\bibitem{mckean1966}
H.~P. McKean, Jr.
\newblock Speed of approach to equilibrium for {K}ac's caricature of a
  {M}axwellian gas.
\newblock {\em Arch. Rational Mech. Anal.}, 21:343--367, 1966.

\bibitem{meng2023}
Fei Meng, Hao Wang, Lihua Min, and Zhengmeng Jin.
\newblock Uniform estimates for the {F}isher information of the {L}andau
  equation for soft potentials.
\newblock {\em J. Math. Anal. Appl.}, 523(1):Paper No. 126992, 21, 2023.

\bibitem{silvestre2017}
Luis Silvestre.
\newblock Upper bounds for parabolic equations and the {L}andau equation.
\newblock {\em J. Differential Equations}, 262(3):3034--3055, 2017.

\bibitem{silvestre2023regularity}
Luis Silvestre.
\newblock Regularity estimates and open problems in kinetic equations.
\newblock In {\em {$A^3N^2M$}: Approximation, Applications, and Analysis of
  Nonlocal, Nonlinear Models: Proceedings of the 50th John H. Barrett Memorial
  Lectures}, pages 101--148. Springer, 2023.

\bibitem{SnelsonIsotropic}
Stanley Snelson.
\newblock Global existence for an isotropic modification of the boltzmann
  equation.
\newblock {\em Journal of Functional Analysis}, 286(12):110423, 2024.

\bibitem{tanaka1978}
Hiroshi Tanaka.
\newblock Probabilistic treatment of the {B}oltzmann equation of {M}axwellian
  molecules.
\newblock {\em Z. Wahrsch. Verw. Gebiete}, 46(1):67--105, 1978/79.

\bibitem{toscani1992}
G.~Toscani.
\newblock New a priori estimates for the spatially homogeneous {B}oltzmann
  equation.
\newblock {\em Contin. Mech. Thermodyn.}, 4(2):81--93, 1992.

\bibitem{toscani2000}
G.~Toscani and C.~Villani.
\newblock On the trend to equilibrium for some dissipative systems with slowly
  increasing a priori bounds.
\newblock {\em J. Stat. Phys.}, 98(5-6):1279--1309, 2000.

\bibitem{villani1998fisherboltzmann}
C.~Villani.
\newblock Fisher information estimates for {B}oltzmann's collision operator.
\newblock {\em J. Math. Pures Appl. (9)}, 77(8):821--837, 1998.

\bibitem{villani1998spatiallyhomogeneous}
C.~Villani.
\newblock On the spatially homogeneous {L}andau equation for {M}axwellian
  molecules.
\newblock {\em Math. Models Methods Appl. Sci.}, 8(6):957--983, 1998.

\bibitem{villani2000fisherlandau}
C.~Villani.
\newblock Decrease of the {F}isher information for solutions of the spatially
  homogeneous {L}andau equation with {M}axwellian molecules.
\newblock {\em Math. Models Methods Appl. Sci.}, 10(2):153--161, 2000.

\bibitem{villani1998Hsolutions}
C\'{e}dric Villani.
\newblock On a new class of weak solutions to the spatially homogeneous
  {B}oltzmann and {L}andau equations.
\newblock {\em Arch. Rational Mech. Anal.}, 143(3):273--307, 1998.

\bibitem{villani2002book}
C\'{e}dric Villani.
\newblock A review of mathematical topics in collisional kinetic theory.
\newblock In {\em Handbook of mathematical fluid dynamics, {V}ol. {I}}, pages
  71--305. North-Holland, Amsterdam, 2002.

\bibitem{villani2025fisher}
C{\'e}dric Villani.
\newblock Fisher information in kinetic theory.
\newblock {\em arXiv preprint arXiv:2501.00925}, 2025.

\bibitem{wu2014}
Kung-Chien Wu.
\newblock Global in time estimates for the spatially homogeneous {L}andau
  equation with soft potentials.
\newblock {\em J. Funct. Anal.}, 266(5):3134--3155, 2014.

\end{thebibliography}

\end{document}